\documentclass[11pt]{article}
\usepackage{amsmath,amsthm,amssymb}
\usepackage{times}
\usepackage{enumerate}
\usepackage[utf8]{inputenc}
\usepackage[english]{babel}
\usepackage{tikz}
\usepackage{dynkin-diagrams}
\usepackage{etex}
\usepackage[T1]{fontenc}
\usepackage[utf8]{inputenx}
\usepackage{etoolbox}
\usepackage[kerning=true,tracking=true]{microtype}
\usepackage{amsmath}
\usepackage{amsfonts}
\usepackage{array}
\usepackage{xstring}
\usepackage{longtable}
\usepackage[listings]{tcolorbox}
\tcbuselibrary{breakable}
\tcbuselibrary{skins}
\usepackage[pdftex]{hyperref}

\usepackage{booktabs}
\usepackage{colortbl}
\usepackage{varwidth}
\usepackage{dynkin-diagrams}
\usepackage{fancyvrb}
\usepackage{xspace}

\usepackage{filecontents}
\usetikzlibrary{decorations.markings}
\usetikzlibrary{decorations.pathmorphing}

\def\titlerunning#1{\gdef\titrun{#1}}
\makeatletter
\def\author#1{\gdef\autrun{\def\and{\unskip, }#1}\gdef\@author{#1}}
\def\address#1{{\def\and{\\\hspace*{18pt}}\renewcommand{\thefootnote}{}%
\footnote {#1}}%
\markboth{\autrun}{\titrun}}
\makeatother
\def\email#1{e-mail: #1}
\def\subjclass#1{{\renewcommand{\thefootnote}{}%
\footnote{\emph{Mathematics Subject Classification (2010):} #1}}}
\def\keywords#1{\par\medskip
\noindent\textbf{Keywords.} #1}

\newtheorem{thm}{Theorem}[section]
\newtheorem{cor}[thm]{Corollary}
\newtheorem{lem}[thm]{Lemma}
\newtheorem{prop}[thm]{Proposition}

\theoremstyle{definition}
\newtheorem{defin}[thm]{Definition}
\newtheorem{rem}[thm]{Remark}
\newtheorem{exa}[thm]{Example}
\newtheorem{notcon}[thm]{Notation}

\numberwithin{equation}{section}

\newcommand{\C}{\mathbb{C}}

\newcommand{\V}{\mathbb{V}}
\newcommand{\W}{\mathbb{W}}
\newcommand{\Z}{\mathbb{Z}}
\newcommand{\mbP}{\mathbb{P}}

\newcommand{\mfg}{\mathfrak{g}}
\newcommand{\mfh}{\mathfrak{h}}

\newcommand{\mfn}{\mathfrak{n}}
\newcommand{\mfp}{\mathfrak{p}}
\newcommand{\mfq}{\mathfrak{q}}

\newcommand{\mcA}{\mathcal{A}}
\newcommand{\mcC}{\mathcal{C}}

\newcommand{\mcF}{\mathcal{F}}
\newcommand{\mcG}{\mathcal{G}}
\newcommand{\mcH}{\mathcal{H}}

\newcommand{\mcP}{\mathcal{P}}
\newcommand{\mcL}{\mathcal{L}}

\newcommand{\mcV}{\mathcal{V}}
\newcommand{\mcO}{\mathcal{O}}
\newcommand{\mcQ}{\mathcal{Q}}

\newcommand{\mcU}{\mathcal{U}}
\newcommand{\mcD}{\mathcal{D}}
\newcommand{\mcK}{\mathcal{K}}

\newcommand{\BP}{\mathbb{P}}

\AtEndDocument{\vfill\eject\batchmode}% Suppress font verbiage at end of output
\usepackage{amssymb}
\usepackage[T1]{fontenc}
\DeclareSymbolFont{script}{U}{eus}{m}{n}
\DeclareMathSymbol{\Wedge}{0}{script}{"5E}

\def\sideremark#1{\ifvmode\leavevmode\fi\vadjust{\vbox to0pt{\vss% the remark
 \hbox to 0pt{\hskip\hsize\hskip1em%                          will appear only
 \vbox{\hsize3cm\tiny\raggedright\pretolerance10000%          on the side
  \noindent #1\hfill}\hss}\vbox to8pt{\vfil}\vss}}}%
\begin{document}

\titlerunning{Cone structures and parabolic geometries}

\title{Cone structures and parabolic geometries}

\author{Jun-Muk Hwang \& Katharina Neusser}

\date{}

\maketitle

\address{J.-M. Hwang: Institute for Basic Science, Center for Complex Geometry; \email{jmhwang@ibs.re.kr} }
\address{K. Neusser: Department of Mathematics and Statistics, Masaryk University; \email{neusser@math.muni.cz}}

\subjclass{Primary 53B99; Secondary 53C56, 14M15}

\begin{abstract}
A cone structure on a complex manifold $M$  is a closed submanifold $\mcC \subset \mbP TM$ of the projectivized tangent bundle which is submersive over $M$.  A conic connection on $\mcC$ specifies a distinguished family of curves on $M$ in the directions specified by $\mcC$. There are two common sources of cone structures and conic connections, one in differential geometry and another in algebraic geometry. In differential geometry, we have cone structures induced by the geometric structures underlying holomorphic parabolic geometries, a classical example of which is the null cone bundle of a holomorphic conformal structure. In algebraic geometry, we have the cone structures consisting of varieties of minimal rational tangents (VMRT) given by minimal rational curves on uniruled projective manifolds.
The local invariants of the cone structures in parabolic geometries are given by the curvature of the parabolic geometries, the nature of which depend on the type of the parabolic geometry, i.e., the type of the fibers of $\mcC \to M$. For the VMRT-structures, more intrinsic invariants of the conic connections
which do not depend on the type of the fiber play important roles.
We study the relation between these two different aspects for the cone structures induced by
parabolic geometries associated with a long simple root of a complex simple Lie algebra.
As an application, we obtain a local differential-geometric version of  the global algebraic-geometric recognition theorem due to Mok and Hong--Hwang. In our local version, the role of rational curves is completely replaced by appropriate torsion conditions on the conic connection.
\end{abstract}

\keywords{cone structures, rational homogeneous space, varieties of minimal rational tangents, Cartan connections, parabolic geometry, filtered manifolds}

\section*{Introduction}
Throughout this article we work in the holomorphic category: manifolds are complex and maps between them are holomorphic. For a complex manifold $M$ we denote by $TM$ and $T^*M$ the holomorphic tangent respectively co-tangent bundle. For a holomorphic vector bundle $E$ over $M$ we write $\mathcal O(E)$ for its sheaf of local (holomorphic) sections. For readability and aesthetic reasons we also sometimes simply identify a vector bundle with its sheaf of local sections and omit $\mcO(_-)$.

A cone structure on a manifold $M$ is a closed submanifold $\mcC \subset \mbP TM$ in the projectivized tangent bundle which is submersive over $M$ (see Definition \ref{d.cone_str}). It specifies a set of distinguished tangent directions on $M$. When it arises in natural geometric problems, it is usually equipped with a conic connection (see Definition \ref{d.conic}), a line subbundle $\mcF \subset T\mcC$, which specifies a set of distinguished curves on $M$ in direction of $\mcC$.

As classical examples of cone structures and conic connections, we have those given by  geodesics in Riemannian geometry or null-geodesics in Lorentzian geometry. The geometric picture becomes more elegant and easier to handle when we consider them in the holomorphic setting. A prototypical example is LeBrun's work \cite{LeBrun}  on null-geodesics in holomorphic conformal geometry. Conformal geometry is a special example of parabolic geometries (see \cite{csbook}), a large class of geometric structures where the methods of natural connections and their curvatures are well-established.    In this regard, it is natural to extend LeBrun's study to  the analogues of null-geodesics associated to other holomorphic parabolic geometries, which is  worked out to some extent in the current paper.   It turns out that  the  cone structure and the conic connection naturally associated with a parabolic geometry can be viewed as another type of parabolic geometry. In particular, this cone structure can be studied
as a special case of the theory of correspondence spaces developed in \cite{CapCores}, as explained in Section \ref{ss.cones_induce_para}. This illustrates the advantage of considering the different types of geometric structures underlying parabolic geometries together in a uniform way as presented in \cite{csbook}.

There is a completely different source of cone structures and conic connections arising from algebraic geometry. On a uniruled projective manifold $X$, one can select a distinguished class of curves called minimal rational curves. The set $\mcC_x \subset \mbP T_x X$ of tangent directions  at $x$ of minimal rational curves through a general point $x \in X$  is called the variety of minimal rational tangents (VMRT) at $x$, which is often a closed submanifold. Then the union of $\mcC_x$ as $x$ varies on a suitable Zariski open subset $M \subset X$ determines a cone structure on $M$,  called a VMRT-structure. It is equipped  with a natural conic connection given by the minimal rational curves.
The theory of VMRT-structures has numerous applications in algebraic geometry as surveyed in
\cite{HwangMSRI},  \cite{HwangICM} and  \cite{Hwang-Mok99}.

Classical examples of uniruled projective manifolds are
 rational homogeneous spaces $G/P.$
When  $P < G$ is a maximal parabolic subgroup determined by a long simple root of a complex simiple Lie group $G$, the VMRT $\mcC_o^{G/P} \subset \mbP T_o G/P$  at the base point $o \in G/P$ turns out to be the unique closed orbit of the $P$-action on $\mbP T_o G/P$ and
the VMRT-structure  $\mcC^{G/P} \subset \mbP T G/P$ is defined on the whole projective manifold $G/P$  via the $G$-action. This is precisely the natural cone structure associated with the flat parabolic geometry of $G/P$. (On the other hand, the VMRT-structure is different from the natural cone structure of the flat parabolic geometry, if $P$ is determined by a short simple root.)
  In this setting, we have the following result, which can be used to characterize $G/P$ with $P$ associated to a long simple root among uniruled projective manifolds of second Betti number 1 in terms of VMRT.

\begin{thm}\label{t.longroot}
Let $P <  G$ be the maximal parabolic subgroup of a complex simple Lie group $G$ determined by a long simple root.
Let $\mcC \subset \mbP TM$ be a VMRT-structure on a Zariski open subset $M \subset X$ of a uniruled projective manifold $X$ with $b_2(X) =1$. Assume that a fiber $\mcC_x \subset \mbP T_x M$ at each $x \in M$ is isomorphic  to
$\mcC_o^{G/P}$ by a projective isomorphism from $\mbP T_x M$ to $\mbP T_o G/P$.
Then the cone structure $\mcC$ is locally isomorphic to $\mcC^{G/P}$. \end{thm}

This was proved by Mok in \cite{Mok} when $G/P$ is of symmetric type or contact type (see Definition \ref{d.symcon}). His method was extended to cover the remaining cases in \cite{HongHwang}. As this is a nice rigidity result on VMRT-structures with important applications in algebraic geometry, one would like to find a similar result, replacing $\mcC^{G/P}_o$ by some other projective submanifold. So it is natural to ask what special properties of $\mcC_o^{G/P}$ have been used in the proof of Theorem \ref{t.longroot}. With some simplification, we can say that the following two properties of $\mcC_o^{G/P}$ were crucial.

\begin{itemize}
\item[(I)] $\mcC_o^{G/P} \subset \mbP T_o G/P$ is essentially determined by its projective fundamental forms.
    \item[(II)] $\mcC_o^{G/P}$ is the fiber of the natural cone structure of the parabolic geometry of type $G/P$. \end{itemize}

Here, we state (I) somewhat vaguely to avoid technicalities. It suffices to know that (I) is a condition coming from the  projective differential geometry of $\mcC_o^{G/P}$ .
The key idea of Mok's proof of  Theorem \ref{t.longroot} is to use (I) to prove that $M$ is covered by minimal rational curves of $X$. On the other hand, (II) implies that the manifold $M$ in Theorem \ref{t.longroot} is equipped with a natural regular normal parabolic geometry with the associated harmonic curvature. By restricting the harmonic curvature to the minimal rational curves on $M$, it is easy to show that the harmonic curvature must vanish, proving the theorem.

Are there projective submanifolds having similar properties? If we consider general Cartan geometries other than parabolic geometries, then (II) can be formulated for any projective submanifold. So there are likely to be many submanifolds satisfying some version of (II). The trouble is (I). Projective submanifolds satisfying an analog of (I)  are very rare. In fact, for a uniruled projective manifold $X$, the Zariski open subset $M \subset X$ where the VMRT-structure is is defined is usually not covered by minimal rational curves of $X$.

One of the motivations of this paper is to give a different proof of Theorem \ref{t.longroot} where  the property (I) is not used at all. This is given in Theorem \ref{t.sMok}, which
is obtained as a byproduct of the following line of investigation.

For the VMRT-structure $\mcC \subset \BP TM$ on the open subset $M$ of a uniruled projective manifold, the fibers $\mcC_x$ and $\mcC_y$ at two distinct points $x \neq y \in M$ are usually not isomorphic. This is very different from geometric structures commonly studied in differential geometry. Such flexibility of choices of fibers in VMRT-structures is actually one of the strong points of the theory of VMRT-structures.
To study cone structures with potentially varying fibers such as VMRT-structures, it is important to use invariants of the structures which can be defined in a uniform way for arbitrary cone structures and conic connections, not depending on specific types of the fibers $\mcC_x$. In Section \ref{s.cone}, we explain two such invariants, the characteristic torsion of a conic connection and the cubic torsion defined for a conic connection with vanishing characteristic torsion. These invariants certainly make sense for the cone structures with conic connections coming from parabolic geometries.
The latter cone structures, however, have also invariants coming from their description as parabolic geometries of type $G/P$, namely, their harmonic curvatures. These are tensors on $M$ which describe certain 
torsion and curvature components of a distinguished class of affine connections of $M$ of which the parabolic geometry is built.

 A natural question is the relation between these two different classes of invariants for cone structures associated with parabolic geometries.
 Our main results,  Theorems \ref{m1} and \ref{t.chi}, describe the relations between these two classes of invariants for the cone structures with conic connection underlying a parabolic geometry of type $G/P$  in Theorem \ref{t.longroot}.  More precisely, Theorem \ref{m1} says that
 the vanishing of characteristic torsion of the conic connection corresponds to the regularity of the induced parabolic geometry on the correspondence space, and Theorem \ref{t.chi} says, excluding two special cases of low dimension of $G/P$, that when the characteristic torsion vanishes, the cubic torsion can be identified with the harmonic curvature of the parabolic geometry on the correspondence space. Theorem \ref{m1} is proved by examining the harmonic curvatures of the correspondence space and the proof of Theorem \ref{t.chi} employs the machinery of Weyl structures in parabolic geometry.

 The paper is organized as follows. In Section \ref{s.cone}, we define conic connections on cone structures and explain the two invariants, the characteristic torsion and the cubic torsion. In Section \ref{s.RHS}, we review the basic theory of parabolic subalgebras and rational homogeneous spaces, with special regard to the natural cone structure on a rational homogeneous space given by a long simple root, which lead to a  nested pair of parabolic subalgebras.
 In Section \ref{s.vmrt}, we give a brief overview of VMRT-structures and explain the vanishing of the characteristic torsions and cubic torsions of the natural conic connections on VMRT-structures.  In Section \ref{s.parabolic}, we study the cone structures associated with parabolic geometries and prove the main results, Theorems \ref{m1}, \ref{t.chi} and \ref{t.sMok}. A brief review of Weyl structures which is needed for the proof of Theorem \ref{t.chi} is given in Subsection \ref{ss.Weyl}.  In the appendix, Section \ref{Appendix}, we list key geometric data for some classes of $G/P$ for the convenience of the reader.

\smallskip{\textbf{Acknowledgement}}\enskip
The authors would like to thank Andreas \v Cap for helpful discussions, in particular about the invariant $\chi_\mcF$. The second author was supported by the grant GX19-28628X of the Czech Science foundation.

\tableofcontents

\section{Cone structures and conic connections}\label{s.cone}
We review the notion of conic connections on cone structures and introduce two fundamental invariants of a conic connection, its characteristic torsion and its cubic torsion.  We start by reviewing some terminology on distributions and filtered manifolds.

\subsection{Filtered manifolds and distributions}\label{ss.filtered}
\begin{defin}\label{d.filtered}
Suppose $M$ is a complex manifold.
\begin{itemize}
\item[(1)] A {\em tangential filtration} of $M$ is a nested sequence of vector subbundles of the tangent bundle $TM$:
\begin{equation}\label{e.filtration}
 TM=T^{-k}M\supset...\supset T^{-1}M,
\end{equation}
 with the convention that $T^{-i}M=TM$ for $i\geq k$ and $T^{-i}M=\{0\}$ for $i\leq 0$. 
 \item[(2)] The \emph{associated graded bundle} of the filtration (\ref{e.filtration}) is the vector bundle on $M$ given by
 $$
\textrm{gr}(TM):=\textrm{gr}_{-k}(TM)\oplus...\oplus \textrm{gr}_{-1}(TM) $$
  where  $\textrm{gr}_{-i}(TM):=T^{-i}M/T^{-i+1}M$ for each $i$.
We write $q_{-i}: T^{-i}M\rightarrow \textrm{gr}_{-i}(TM)$ for the natural projections.
\item[(3)] A complex manifold $M$ equipped with a filtration (\ref{e.filtration}) is called a \emph{filtered manifold}, if the filtration is compatible with the Lie bracket of vector fields, that is,
$$
[\mcO(T^{-i}M), \mcO(T^{-j}M)]\subset \mcO(T^{-(i+j)}M) \quad \textrm{ for all } i,j >0.
$$
\item[(4)] A \emph{(local) isomorphism between filtered manifolds}
$(M, \{T^{-i}M\})$ and $(M', \{T^{-i}M'\})$ is a (local) biholomorphism $\phi: M\rightarrow M'$ whose tangent map sends $T^{-i}M$ to $T^{-i}M'$ for each $i$.
\end{itemize}
\end{defin}

\begin{defin}\label{d.Levi}
Given a filtered manifold $(M, \{T^{-i}M\})$, the Lie bracket induces the vector bundle map $\mathcal L: \textrm{gr}(TM)\otimes\textrm{gr}(TM)\rightarrow\textrm{gr}(TM)$
whose restriction on the component
\begin{equation}\label{e.Levi}
\mathcal L: \textrm{gr}_{-i}(TM)\otimes\textrm{gr}_{-j}(TM)\rightarrow\textrm{gr}_{-(i+j)}(TM),
\end{equation}
is defined by $\mathcal L(q_{-i}(\xi), q_{-j}(\eta))=q_{-(i+j)}([\xi,\eta])$ for sections $\xi$ and $\eta$ of $T^{-i}M$ and $T^{-j}M$ respectively.
The bundle map \eqref{e.Levi} is  called the \emph{Levi bracket}  of the filtered manifold $(M, \{T^{-i}M\}).$ For each $x \in M$, the graded nilpotent Lie algebra $(\textrm{gr}(T_xM), \mathcal L_x)$ is called the \emph{symbol algebra} of $(M, \{T^{-i}M\})$ at $x\in M$.
\end{defin}

The significance of the notion of filtered manifolds emerges from the study of distributions: 

\begin{defin}\label{d.distribution}
 Suppose $M$ is a complex manifold and $\mcH\subset TM$ a \emph{distribution} on $M$, by which we mean a (holomorphic) vector subbundle $\mcH$ of $TM$.
The \emph{weak derived flag} of $\mcH$ is the sequence of subsheaves of $\mcO(TM)$ given by
\begin{equation}\label{weak_derived_flag}
\mcH^{-1}\subset \mcH^{-2}...\subset \mcO(TM),
\end{equation}
where $\mcH^{-1}:=\mcO(\mcH)$ and $\mcH^{-i}$ is defined inductively as the saturated subsheaf of $\mcO(TM)$ generated by $\mcH^{-i+1}$ and $[\mcH^{-1}, \mcH^{-i+1}]$.
Replacing $M$ by an open dense subset of $M$ if necessary, \eqref{weak_derived_flag} induces the structure of a filtered manifold on $M$ given by
\begin{equation}\label{weak_derived_flag_regular}
T^{-1}M\subset ...\subset T^{-k}M\subset T^{-(k+1)}M:=TM,
\end{equation}
where $k>0$ is the smallest integer such that $\mcH^{-i}=\mcH^{-k}$ for all $i\geq k$ and $T^{-i}M\subset TM$ is the vector subbundle such that $\mcO(T^{-i}M)=\mcH^{-i}$ for $i\leq k$.
We say that $\mcH$ is  \emph{bracket-generating}, if $\mcH^{-k}=\mcO(TM)$, or equivalently,  $\textrm{gr}_{-1}(T_xM)=T^{-1}_xM$ generates $\textrm{gr}(T_xM)$ as a Lie algebra for a general point $x\in M$.
If $T^{-k}M\subsetneq TM$, then $T^{-k}M$ is involutive and $M$ is foliated by manifolds equipped with bracket generating distributions.
\end{defin}

We will be mainly interested in filtered manifolds with constant symbol:

\begin{defin} \label{d.Fr}
Let $\mfn$ be a nilpotent graded Lie algebra and let $\textrm{Aut}_{\textrm{gr}}(\mfn)$ be the group of graded Lie algebra automorphisms of $\mfn$.
\begin{itemize} \item[(1)] A filtered manifold $(M, \{T^{-i}M\})$ has \emph{constant symbol of type} $\mfn$,
if its symbol algebra at each point $x\in M$ is isomorphic to $\mfn$.
\item[(2)] In (1),  denote by
$\textrm{Fr}(\textrm{gr}(T_xM))$ the set of graded  Lie algebra isomorphisms $\mfn\rightarrow\textrm{gr}(T_xM)$ and define $$\textrm{Fr}(\textrm{gr}(TM)):=\sqcup _{x\in M}\textrm{Fr}(\textrm{gr}(T_xM)),$$ which is a principal $\textrm{Aut}_{\textrm{gr}}(\mfn)$-bundle on $M$ called the \emph{frame bundle} of $(M, \{T^{-i}M\})$. The principal right action of $\textrm{Aut}_{\textrm{gr}}(\mfn)$ on the fibers $\textrm{Fr}(\textrm{gr}(T_xM))$ is given by pre-composition of maps.
\item[(3)] If $T^{-1}M=TM$, then $\textrm{Fr}(\textrm{gr}(TM))$ is the usual frame bundle $\textrm{Fr}(TM)$ of $M$. \end{itemize}
\end{defin}

\subsection{Cone structures}\label{ss.cone}
We use the following terminology in projective differential geometry.

\begin{notcon}
Let $\mbP \V$ be the projectivization of a vector space $\V$. Denote by $[v] \in \mbP \V$ the point corresponding to a nonzero vector $v \in\V$. Let $Z \subset \mbP \V$ be a closed submanifold.
\begin{itemize}
\item[(1)] The \emph{affine cone} of $Z$ is denoted by $$\widehat Z :=\{0 \neq v\in\V: [v]\in Z\} \cup \{ 0 \}.$$ In particular, the 1-dimensional subspace of $\V$ corresponding to a point $w \in \mbP \V$ is $\widehat w \subset \V$.
\item[(2)] The linear automorphism group of $\widehat Z$, denoted by
$\textrm{Aut}(\widehat Z),$  is the group of linear automorphisms of $\V$ preserving $\widehat Z$.
\item[(3)]
The \emph{affine tangent space} of $Z$ at $z\in Z$ is
$$\widehat{T}_{z} Z :=T_v\widehat Z \quad\quad \textrm{ for some } 0\neq v\in\hat z\subset\V.$$
\end{itemize}
\end{notcon}

\begin{defin}\label{d.cone_str}
Let $M$ be a complex manifold of dimension $n+1$ and let $\mathbb P TM$ its projectivized tangent bundle.
\phantom{} \quad
\begin{itemize}
\item[(1)] A  \emph{cone structure} on $M$ is a  closed complex submanifold $\mcC\subset \mathbb P(TM)$ such that the natural projection $p:\mcC\rightarrow M$ is a submersion in the sense that
$Tp: T_u \mcC \to T_{p(u)}M$ is surjective for every $u \in \mcC$. In particular, every fiber $\mcC_x = p^{-1}(x)$ is a  closed submanifold of the $n$-dimensional projective space $\mathbb P T_xM$.
\item[(2)] Suppose $\mcC\subset\mathbb P(TM)$ and $\mcC'\subset\mathbb P(TM')$ are two cone structures on manifolds $M$ and $M'$ respectively. Then $\mcC$ is \emph{locally isomorphic} to $\mcC'$, if there exist nonempty connected open subsets $U\subset M$ and $U'\subset M'$ equipped with a biholomorphic map $\phi: U\rightarrow U'$ such that $\mathbb P(T\phi):  \mathbb P(TU)\rightarrow  \mathbb P(TU')$
induces a biholomorphic map from  $\mcC \cap \mathbb P(TU) $ to  $\mcC' \cap \mathbb P(TU')$. \end{itemize}
\end{defin}

In this article we are particularly interested in cone structures whose fibers are all isomorphic as projective submanifolds:

\begin{defin}\label{d.G-structure}
Let $\mathbb V$ be a vector space of dimension $n+1$ and $Z \subset\mathbb P\mathbb V$ a closed submanifold. A cone structure $\mcC\subset \mathbb P(TM)$ on a complex manifold $M$ is called $Z$--\emph{isotrivial}, if for any $x\in M$ the fiber $\mcC_x\subset\mbP T_xM$ is isomorphic to $Z \subset\mbP\V$ as a projective submanifold. 
Note that any $Z$--isotrivial cone structure $\mcC$ defines a reduction of the structure group of the frame bundle of
$M$ to $\textrm{Aut}(\widehat Z)$ corresponding to the group inclusion $\textrm{Aut}(\widehat Z)\subset \textrm{GL}(\V)$.
\end{defin}

\begin{exa}\label{ex.flat}
Let $\V$ be a vector space and $Z \subset \mbP \V$ be a  closed submanifold. Then $\mcC:=\V\times Z \subset \V\times\mbP\V=\mbP(T\V)$ is a cone structure on $\V$.
It is called the \emph{flat cone structure} on $\V$ with fiber $Z \subset \mbP\V$. 
\end{exa}

\begin{exa}\label{ex.conformal}
Suppose $M$ is a complex manifold of dimension $n+1$.
A \emph{(holomorphic) conformal structure} on $M$ is a $\textrm{CO}(n+1)$-structure on $M$, where $\textrm{CO}(n+1)$ denotes the group of linear conformal transformations of complex Euclidean space $\C^{n+1}$.
To any conformal structure we can associate its bundle of null cones $\mcC\subset \mbP TM$, which is a $Z$-isotrivial cone structure, where $Z \subset \C\mbP^n$ is the standard nonsingular quadric. Conversely,
the reduction of ${\rm Fr}(TM)$ associated to any $Z$-isotrivial cone structure $\mcC\subset \mbP (TM)$ defines a conformal structure, since $\textrm{Aut}(\widehat Z)\cong\textrm{CO}(n+1)$.
These assignments are evidently inverse to each other.
\end{exa}

\begin{defin}\label{d.subordinate}
A cone structure $\mcC\subset\mathbb P TM$ on a filtered manifold $(M, \{T^{-i}M\})$  is called
\emph{subordinate} to the tangential filtration $\{T^{-i}M\}$, if $\mcC\subset \mathbb P(T^{-1}M)$. Note, if $\mcC$ is $Z$--isotrivial for some $Z\subset \mathbb P \W\subset \mathbb P \V$ with $\dim \W=\textrm{rank}(T^{-1}M)$, then $\mcC$ defines a reduction of the structure group of the frame bundle  $\textrm{Fr}(T^{-1}M)$ of the vector bundle
$T^{-1}M$ corresponding to the inclusion $\textrm{Aut}(\widehat Z)\subset \textrm{GL}(\W)$. 
\end{defin}

\begin{defin}\label{d.cone_filtration}
Suppose $\mcC\subset \mathbb P TM$ is a cone structure on a complex manifold $M.$
\begin{itemize}
\item[(1)] For each $u \in \mcC$, define $$\mcD^{-1}_{u}:=(Tp)^{-1}(\hat u) \mbox{ and }
    \mcD^{-2}_{u} := (Tp)^{-1}(\widehat{T}_{u} \mcC_{p(u)}).$$
    They determine vector subbundles \begin{equation}\label{cone_filtration}
\mcD^{-1}\subset \mcD^{-2}\subset T\mcC,
\end{equation}
where, by construction, $\mcD^{-1}$ contains the vertical subbundle $\mcV$ of $p: \mcC\rightarrow M$.
Note that, if $r$ is the dimension of the fibers of $p:\mcC\rightarrow M$, then $\mcV$ has rank $r$, $\mcD^{-1}$ has rank $r+1$, and $\mcD^{-2}$ has rank $2r+1$.
\item[(2)] Inductively define saturated subsheaves   $$\mcD^{-i}=[\mcV, \mcD^{-i+1}]+\mcD^{-i+1}$$ of $TM$ for any $i \geq 3$. They are not necessarily locally free.
\end{itemize} 
\end{defin}

The following can be proved by a direct generalization of the proof of Proposition 2 of \cite{Hwang-Mok04}.

\begin{lem}\label{l.HM04}
In Definition \ref{d.cone_filtration},
 if $\mcC_x\subset \mathbb PT_xM$ is linearly non-degenerate
(that is, it is not contained in a hyperplane) for a general point $x \in M$, then $\mcD^{-i}=T\mcC$ for sufficiently large $i$ on a Zariski open subset of $\mcC$. \end{lem}

The next proposition was proved in Proposition 1 of \cite{Hwang-Mok04}.

\begin{prop}\label{p.HM04}
Suppose $\mcC\subset \mathbb P(TM)$ is a cone structure on a complex manifold $M$. Then
$\mcD^{-2}$ coincides with the first derived system of $\mcD^{-1}$ (that is, $\mcD^{-2}$ is generated by the sheaf of local sections of $\mcD^{-1}$ and its Lie brackets)
and hence \eqref{cone_filtration} gives $\mcC$ the structure of a filtered manifold.
\end{prop}

\begin{rem}
The filtration $\{\mcD^{-i}\}$ of $T\mcC$ is in general not compatible with the Lie bracket of vector fields, as we see in Subsection \ref{ss.cones_induce_para}.
\end{rem}

\subsection{Conic connections and their characteristic torsion}\label{ss.conic}

\begin{defin}\label{d.conic}
A \emph{conic connection} on a cone structure $\mcC\subset\mathbb P TM$ is a splitting of the short exact
sequence of vector bundles arising from Definition \ref{d.cone_filtration} (1)
\begin{equation*}
0\rightarrow \mcV\rightarrow \mcD^{-1}\rightarrow \mcD^{-1}/\mcV\rightarrow 0,
\end{equation*}
equivalently, a line subbundle $\mcF\subset \mcD^{-1}$ such that $\mcD^{-1}=\mcF\oplus \mcV$.
\end{defin}

Note that the projections of the integral curves of a conic connection $\mcF\subset \mcD^{-1}\subset T\mcC$ to $M$ give rise to a family of complex curves on $M$, with exactly one curve through any point $x\in M$ in any direction belonging 
to $\mcC_x$. Conversely, any such family lifting to a holomorphic foliation of rank $1$ on $\mcC$ defines a conic connection on $\mcC$.

\begin{exa}\label{ex.path.geom}
If $\mcC=\mbP TM$, then $T{\mcC}=\mcD^{-2}$ and a conic connection is classically called a \emph{path geometry} on $M$. In the holomorphic category the curves in $M$ of any path geometry can be realized locally as the geodesics of a holomorphic torsion-free affine connection \cite[Prop. 1.2.1]{LeBrunThesis} and hence a path geometry is also known to be equivalent to a \emph{projective structure} on $M$.
\end{exa}

Note that an immediate consequence of Proposition \ref{p.HM04} and the integrability of $\mcV$ is:

\begin{cor}\label{c.isom} Suppose $\mcC\subset\mbP TM$ is a cone structure on a complex manifold $M$ equipped with a conic connection $\mcF$.
Then the Levi bracket of \eqref{cone_filtration} induces an isomorphism
\begin{equation}
\mcF\otimes \mcV\cong \emph{gr}_{-2}(\mcD).
\end{equation}
\end{cor}

\begin{defin}\label{d.characteristic}
Let  $\mcF$ be a conic connection on a cone structure $\mcC\subset \mathbb P TM$. The \emph{characteristic torsion} of $\mcF$ is the section $\tau_{\mcF}$ of the vector bundle $\mcF^*\otimes\textrm{gr}_{-2}(\mcD)^*\otimes T\mcC/\mcD^{-2}$, which is the negative of the corresponding component of the Levi bracket $\mcL$ of \eqref{cone_filtration}:
$$
\begin{array}{ccc} \mcF  \otimes \textrm{gr}_{-2} (\mcD) & \stackrel{ \tau_{\mcF}}{\longrightarrow} &   T\mcC/\mcD^{-2} \\ \cap & & \| \\ \mcD^{-1} \otimes \textrm{gr}_{-2}( \mcD) & \stackrel{- \mcL}{\longrightarrow} &  T\mcC/\mcD^{-2}.  \end{array}  $$
A conic connection is said to be \emph{characteristic}, if $\tau_{\mcF} =0.$ Equivalently, a conic connection is a characteristic connection if and only if
$$
[\mcF, \mcD^{-2}]\subset \mcD^{-2}.
$$
\end{defin}

\begin{lem}\label{l.-3}
For a cone structure $\mcC\subset \mathbb P TM$ with a characteristic conic connection $\mcF$, one  has $[\mcF, \mcD^{-3}]\subset \mcD^{-3}$ and $[\mcD^{-2}, \mcD^{-2}]\subset \mcD^{-3}$.
\end{lem}
\begin{proof}
Recall  $\mcD^{-2}=[\mcV, \mcF]+\mcD^{-1}$ and $\mcD^{-3}=[\mcV, \mcD^{-2}]+\mcD^{-2}$. Let  $f$ be a local section of $\mcF$, and let $v, w$ be local sections of $\mcV$. The Jacobi identity implies that
\begin{align*} [f, [v,[w,f]]]=-[v, [[w,f], f]]-[[w,f], [f,v]].
\end{align*}
Since $\mcF$ is characteristic, the first term on the right-hand side lies in $\mcD^{-3}$, which shows that $[\mcF, \mcD^{-3}]\subset \mcD^{-3}$ if and only if $[\mcD^{-2}, \mcD^{-2}]\subset \mcD^{-3}$.
Applying the Jacobi identity to $[[w,f], [f,v]]$ shows that
\begin{eqnarray*}
[[w,f], [f,v]] &=& -[[f,[f,v]], w]+[f, [[f,v], w]] \\ &=& -[[f,[f,v]], w]-[f, [[v,w], f]]+[f, [v, [w, f]]].
\end{eqnarray*}
Inserting this identity into the previous one, we obtain
$$2 [f, [v,[w,f]]]=[v, [[w,f], f]]+[[f,[f,v]], w]+[f, [[v,w], f]].$$
Since $\mcF$ is characteristic and the vertical bundle $\mcV$ integrable (hence $[\mcV, \mcV]\subset \mcV$), every term on the right-hand side is a section of $\mcD^{-3}$. Hence,  $[\mcF, \mcD^{-3}]\subset \mcD^{-3}$ as claimed.
\end{proof}

\begin{defin}\label{d.tau_C}
Suppose $\mcC\subset \mathbb P(TM)$ is a cone structure and denote by $${\rm II}: \mcV \rightarrow  \textrm{gr}_{-2}(\mcD)^*\otimes T\mcC/\mcD^{-2}$$ the appropriate component of the Levi bracket of \eqref{cone_filtration}. Consider the component $\mcL: \mcD^{-1}\otimes  \textrm{gr}_{-2}(\mcD)\rightarrow  T\mcC/\mcD^{-2}$ of the Levi bracket of \eqref{cone_filtration}.
Then the following composition of maps, where the second map denotes the natural projection,
\begin{equation}\label{e.tau}
\mcD^{-1}\stackrel{-\mcL}{\rightarrow} \textrm{gr}_{-2}(\mcD)^*\otimes T\mcC/\mcD^{-2}\rightarrow (\textrm{gr}_{-2}(\mcD)^*\otimes T\mcC/\mcD^{-2})/\textrm{Im}({\rm II}),
\end{equation}
descends to a section $\tau^{\mcC}$ of
\begin{equation}\label{e.quotient}
(\mcD^{-1}/\mcV)^*\otimes (\textrm{gr}_{-2}(\mcD)^*\otimes T\mcC/\mcD^{-2}))/\textrm{Im}({\rm II}),
\end{equation}
which we call the \emph{characteristic torsion of $\mcC$}.
\end{defin}

The significance of $\tau^\mcC$ is the following:

\begin{prop}\label{p.torsion}
For a cone structure $\mcC\subset \mathbb P(TM),$  assume that ${\rm II}: \mcV\rightarrow \emph{gr}_{-2}^*(\mcD)\otimes T\mcC/\mcD^{-2}$ is injective at general points of $\mcC$.
\begin{enumerate}
\item[(1)] Then $\mcC$ admits at most one characteristic conic connection.
\item[(2)] Assume $\mcC$ admits a conic connection. Then $\tau^\mcC=0$ if and only if $\mcC$ admits a characteristic conic connection (that is, a conic connection $\mcF$ with $\tau_\mcF=0$).
\end{enumerate}
\end{prop}
\begin{proof}
For (1), we can translate the proof of Proposition 3 in \cite{Hwang-Mok04} into our terminology.  Suppose $f$ and $f'$ are local sections of two conic connections $\mcF$ and $\mcF'$ respectively. Then $f=hf'+v$ for a unique local section $v$ of $\mcV$ and a holomorphic function $h$. If both connections are characteristic, then one must have
$[v, \mcD^{-2}]\subset \mcD^{-2}$. Hence, ${\rm II}(v)=0$ and injectivity of ${\rm II}$ implies $v=0$.

For (2), if $\mcC$ admits a characteristic conic connection, then clearly $\tau^{\mcC}=0$. Conversely, assume now that $\tau^{\mcC}=0$ and let $\mcF$ be any conic connection, which exists by assumption.
Then for any local section $f$ of $\mcF$ there exists a unique (by injectivity of ${\rm II}$) local section $v$ of $\mcV$ such that $\mcL(f, q_{-2}(\xi))=\mcL(v, q_{-2}(\xi))$ for all section $\xi$ of $\mcD^{-2}$.
Hence, $f-v$ defines locally a characteristic conic connection on $\mcC$. By (1), the so constructed local characteristic conic connections patch together to a global one, which completes the proof.
\end{proof}

\begin{lem}\label{FF.injective} 
For a cone structure $\mcC\subset \mathbb P(TM),$  assume that each component of the fiber 
$\mcC_x \subset \mathbb P(T_x M)$ is different from a linear subspace for each $x \in M$. Then 
${\rm II}: \mcV\rightarrow \emph{gr}_{-2}^*(\mcD)\otimes T\mcC/\mcD^{-2}$ in Proposition 
\ref{p.torsion}  is injective at general points of $\mcC$.  
\end{lem}

\begin{proof}
We may assume that $\mcC_x$ is connected. Suppose the kernel of ${\rm II}$ has positive rank at general points of $\mcC$.
By Proposition 2 of \cite{Hwang-Mok04}, this means that the Gauss map of $\mcC_x$ for
general $x \in M$ has positive-dimensional fiber. This implies that $\mcC_x$ is a linear subspace
by Theorem 4.3.2 of \cite{IL}. 
\end{proof}

\begin{exa}
If $\mcC\subset \mbP TM$ is the bundle of null cones of a conformal structure as in Example \ref{ex.conformal}, then a conic connection on $\mcC$ is also called a \emph{conformal connection} in \cite[p. 216]{LeBrun}.
Then there exists a unique characteristic conic connection defined by the null geodesics of the conformal structure.
\end{exa}

A consequence of the existence of a characteristic conic connection is the following result from Theorem 6.2 of \cite{HwangMSRI}.

\begin{prop}\label{p.Hw12}
Let $(M, \{ T^{-i}M \})$ be a filtered manifold and $\mcC \subset \mbP T^{-1} M$
be a cone structure subordinate to the filtration.
If $\mcC$  admits a characteristic conic connection (i.e. $\tau^{\mcC} =0$), then the Levi bracket $\mcL: T^{-1} M \otimes T^{-1} M \to T^{-2} M/T^{-1} M$ satisfies $\mcL (u, v) =0$ for any $u, v \in T^{-1}_x M, x \in M,$ with $u \in \widehat{\mcC}_x$ and $v \in \widehat{T}_u \mcC_x$. \end{prop}

\subsection{Cubic torsion of a characteristic conic connection}\label{ss.cubic}
To a characteristic connection $\mcF$ on a cone structure $\mcC\subset\mbP TM$, one can associate a local invariant $\chi_\mcF$, which is a section of
$$S^3\mcF^*\otimes \textrm{Hom}_0(\mcV, \textrm{gr}_{-2}(\mcD)),$$ where $\textrm{Hom}_0(\mcV, \textrm{gr}_{-2}(\mcD))$ denotes the subbundle of
$\textrm{End}(\mcV, \textrm{gr}_{-2}(\mcD))$ of trace-free homomorphisms with respect to the isomorphism $\mathcal L: \mcF\otimes \mcV\rightarrow \textrm{gr}_{-2}(\mcD)$ in Corollary \ref{c.isom}.  To define it, we need some auxiliary definitions as follows.

\begin{defin}\label{d.auxilliary}
Suppose $f$ is a local non-vanishing section of a characteristic connection $\mcF$ on a cone structure $\mcC \subset \mbP TM$. Since $\mcF$ is assumed to be characteristic, for any section $v$ of $\mathcal V$
there exists a unique local section $w(f, v)$ of $\mcV$ such that
\begin{equation}\label{w(f,v)}
\mathcal L(f, w(f,v))=\frac{1}{2} q_{-2}([f, [f,v]]),
\end{equation}
where $q_{-2}: \mcD^{-2}\rightarrow \textrm{gr}_{-2}(\mcD)$ is the projection in Definition \ref{d.filtered} (2).
Define $\tilde\chi(f): \mathcal O(\mcV)\rightarrow\mathcal O (\textrm{gr}_{-2}(\mcD))$ by
\begin{equation}\label{tildechi}
\tilde \chi (f)(v):=-q_{-2}([f,[f, [f,v]-\frac{3}{2}w(f,v)]]).
\end{equation}
\end{defin}

\begin{lem} \label{l.invariant}
In Definition \ref{d.auxilliary}, for a local non-vanishing holomorphic function $h$ on $\mcC$ and a local section $v$ of $\mcV$, the map $\tilde\chi(f)$ has the following properties.
\begin{enumerate}
\item[(1)]  $w(f,hv)=hw(f,v)+(f\cdot h) v.$
\item[(2)] $\tilde\chi(f) (hv) = h \tilde\chi(f) (v),$ which implies that $\tilde\chi(f)$ is a tensor, that is, it defines a section of $\emph{Hom}(\mcV, \emph{gr}_{-2}(\mcD))$.
\item[(3)] $ w(hf, v)=hw(f,v)+\frac{1}{2}(f\cdot h)v.$
\item[(4)] $\tilde\chi(hf)\equiv h^3\tilde\chi(f)\mod \mcL(f, _-)$.
\end{enumerate}
\end{lem}
\begin{proof}
 (1) follows from $$[f,[f,hv]]= h[f,[f,v]]+2(f\cdot h)[f,v]+f\cdot (f\cdot h) v.$$ Bracketing this expression again with $f$ shows
that $$q_{-2}([f,[f,[f,hv]]])=q_{-2}(h[f,[f,[f,v]]]+3 (f\cdot h)[f,[f,v]]+3 f\cdot (f\cdot h)[f,v]).$$
Then (1) implies that
\begin{align*}
q_{-2}([f,[f,w(f,hv)]])&=q_{-2}(h[f,[f,w(f,v)]]+2 (f\cdot h)[f, w(f,v)]\\
&\quad\quad+(f\cdot h)[f,[f,v]]+2 f\cdot (f\cdot h)[f,v])\\
&=q_{-2}(h[f,[f,w(f,v)]]+2 (f\cdot h)[f,[f,v]]+2 f\cdot (f\cdot h)[f,v]),
\end{align*}
 which proves (2). Moreover, one computes straightforwardly that \begin{align*} \frac{1}{2} q_{-2}([hf, [hf,v]]) &= \frac{1}{2} q_{-2}(h^2 [f,[f,v]]+h (f\cdot h) [f,v]) \\ &=\mathcal L(hf, hw(f,v)+\frac{1}{2}(f\cdot h)v),\end{align*} which proves (3).
Using this shows that
\begin{eqnarray*} \lefteqn{q_{-2}([hf,[hf, w(hf,v)]]) \equiv } \\ & &  h^3q_{-2}([f,[f, w(f,v)]])+2 h^2(f\cdot h)q_{-2}([f,[f,v]]) \mod \mcL(f,v), \end{eqnarray*}
and hence $$q_{-2}([hf, [hf, [hf,v]]])\equiv h^3q_{-2}([f, [f, [f,v]]])+3 h^2(f\cdot h)q_{-2}([f,[f,v]]) \mod \mcL(f,v)$$
implies (4).
\end{proof}

\begin{defin}\label{d.cubic} Let $\mcC\subset \BP TM$ be a cone structure with a characteristic conic connection $\mcF$.
In Definition \ref{d.auxilliary}, let  $$\chi_\mcF(f) := \tilde\chi (f) \  \mod \mcL (f,  \cdot ) $$ be the trace-free part of $\tilde\chi(f)$.  By Lemma \ref{l.invariant}, the association $f\otimes f\otimes f\mapsto \chi_\mcF(f)$ determines a section $\chi_\mcF$ of $S^3\mcF^*\otimes \textrm{Hom}_0(\mcV, \textrm{gr}_{-2}(\mcD)).$ We call $\chi_{\mcF}$
 the \emph{cubic torsion} of the characteristic connection $\mcF$.
\end{defin}

We shall see in Section \ref{s.vmrt} that there are many natural examples  arising from algebraic geometry of cone structures with conic connections $\mcF$ with vanishing $\tau_\mcF$ and $\chi_\mcF$. On the other hand, we shall see in Section \ref{s.parabolic} that for a certain class of isotrivial cone structures the existence of a conic connection $\mcF$ with vanishing $\tau_\mcF$ and  $\chi_\mcF$ determines the local isomorphism type of the cone structure.

\section{Parabolic subalgebras and rational homogeneous spaces}\label{s.RHS}

\subsection{Parabolic subalgebras and their homology groups}

Recall that a subalgebra $\mfp$ of a complex semisimple Lie algebra $\mfg$ is said to be \emph{parabolic}, if it contains a maximal solvable subalgebra (called \emph{Borel subalgebra}) of $\mfg$.
 Then the following is well-known, see e.g. \cite[Theorem 3.2.1]{csbook}:

\begin{prop}\label{p.rootparabolic}
Let $\mfg$ be a complex semisimple Lie algebra.
Fix  a Cartan subalgebra $\mfh$ of  $\mfg$
and let $\Delta^0$ be a subset of simple roots in the set of roots $\Delta$ of $\mfg$ with $\mfh$.
 For any subset
$\Sigma\subset\Delta^0$ and $\alpha\in\Delta,$ write $\textrm{ht}_\Sigma(\alpha)$ for the sum of all the coefficients of elements in $\Sigma$ in the expression of
$\alpha$ as a linear combination of simple roots. \begin{itemize} \item[(1)]
Any subset $\Sigma\subset \Delta^0$ gives rise to a graded Lie algebra structure on $\mfg$
\begin{equation}\label{grading}
\mfg=\mfp_{-k}\oplus \cdots \oplus \mfp_0\oplus \cdots \oplus\mfp_k\quad\quad [\mfp_i,\mfp_j]\subset\mfp_{i+j},
\end{equation}
where $\mfp_i:=\oplus_{\textrm{ht}_{\Sigma}(\alpha)=i}\,\mfg_{\alpha}$ for $i\neq0$, $\mfp_0:=\mfh\oplus\bigoplus_{\textrm{ht}_{\Sigma}(\alpha)=0}\mfg_{\alpha}$, and $\mfp_{-1}$ generates the subalgebra $\mfp_-:=\bigoplus_{i\geq 1}\mfp_{-i}$. The largest integer $k$ with $\mfp_k \neq 0$ is called the {\em depth} of $\mfp$  (of $\Sigma$).
\item[(2)] The subalgebra $\mfp= \mfp_{\Sigma}:=\mfp_0\oplus\mfp_+$ of $\mfg$ is parabolic, where $\mfp_+:=\bigoplus_{i\geq 1}\mfp_i$ is its nilradical and $\mfp_0\leq\mfp$, called its Levi subalgebra, is reductive.
   \end{itemize}
    Conversely, any parabolic subalgebra of $\mfg$ is conjugate to  $\mfp_{\Sigma}$ for some $\Sigma \subset \Delta^0$.
    \end{prop}
 
 \begin{notcon}\label{n.Dynkin}
 A parabolic subalgebra $\mfp=\mfp_\Sigma\leq \mfg$ as in Proposition \ref{p.rootparabolic} (or its conjugacy class) is 
 denoted by the Dynkin diagram of $\mfg$ with vertices corresponding to roots in $\Sigma$ replaced by crosses.
 \end{notcon} 

\begin{notcon}\label{n.G/P} 
Suppose $\mfp=\mfp_\Sigma$ is parabolic subalgebra of a semisimple Lie algebra.
\begin{itemize}
\item We write $\mfp_0^{ss}$ for the semisimple part of its Levi subalgebra $\mfp_0$  and $\mathfrak{z}(\mfp_0)$ for the center of $\mfp_0$.
\item We denote the filtered Lie algebra structure on $\mfg$ induced by $\mfp=\mfp_\Sigma$ by
 \begin{equation}\label{filtration}
\mfg=\mfp ^{-k}\supset \cdots \supset \mfp^{-1}\supset \mfp^0\supset \mfp^{1}\supset \cdots \supset\mfp^k \quad\quad [\mfp^i, \mfp^j] \subseteq \mfp^{i+j}
\end{equation} where   $\mfp^j$ is a $\mfp$-module defined by $\mfp^j:=\oplus_{i\geq j}\mfp_i$.
\end{itemize} 
\end{notcon}

\begin{lem}\emph{[see e.g. \cite[Prop. 3.1.2]{csbook}]}\label{l.basic_isos}
In Proposition \ref{p.rootparabolic}, the Killing form of $\mfg$ induces:
\begin{itemize}
\item $\mfp_0$-modules isomorphisms  $\mfp_{-i}^*\cong \mfp_i$ and hence $\mfp_-^*\cong \mfp_+$;
\item a $\mfp$-module isomorphism $(\mfg/\mfp)^*\cong\mfp_+$.
\end{itemize}
\end{lem}

\begin{defin}\label{d.homology} Regarding $\mfg$ as a $\mfp_+$-module by the adjoint representation, there is complex by boundary operators \begin{equation}\label{e.complex} \cdots
\Wedge^{\ell-1} \mfp_+\otimes\mfg \stackrel{\,\,\,\partial^*}{\leftarrow} \Wedge^{\ell} \mfp_+\otimes\mfg
\stackrel{\,\,\,\partial^*}{\leftarrow} \Wedge^{\ell +1}\mfp_+\otimes\mfg \cdots
\end{equation}
computing the \emph{Lie algebra homology $H_{\ell}(\mfp_+,\mfg)$ of $\mfp_+$ with values in~$\mfg$}. By Lemma \ref{l.basic_isos}, as $\mfp_0$-modules 
$\Wedge^\ell\mfp_-^*\otimes\mfg\cong \Wedge^\ell\mfp_+\otimes\mfg$ and hence the complex computing the cohomology of $\mfp_-$ with values in the $\mfp_-$-module $\mfg$ gives rise 
to a complex of coboundary operators of the form
$$ \cdots
\Wedge^{\ell-1} \mfp_+\otimes\mfg \stackrel{\,\,\,\partial}{\rightarrow} \Wedge^{\ell} \mfp_+\otimes\mfg
\stackrel{\,\,\,\partial}{\rightarrow} \Wedge^{\ell +1}\mfp_+\otimes\mfg \cdots.
$$ The formulae for these operators can be found in Sections 2.1.9 and 3.3.1 of \cite{csbook}.
\end{defin}

Kostant showed in \cite[Theorem 5.14]{Kostant} (see also \cite[Corollary 3.1.11]{csbook}) that the operators in Definition \ref{d.homology} give rise to an algebraic Hodge decomposition of $\Wedge^*\mfp_+\otimes\mfg$ as follows.

\begin{prop}\label{p.Kostant}
In Definition \ref{d.homology}, let $G$ be the adjoint group of $\mfg$ and let $P \subset G$ (resp. $P_+, P_0$) be the connected subgroup corresponding to the subalgebra $\mfp \subset \mfg$ (resp. $\mfp_+, \mfp_0$).
\begin{itemize}
 \item[(1)]
The boundary operators $\partial^*$ are $P$-equivariant and the homology group
$H_\ell(\mfp_+, \mfg)$ is a completely reducible $P$-module. In particular, it has a natural $P_0$-module structure.
\item[(2)]
The coboundary operator $\partial$ is $P_0$-equivariant and there is a natural $P_0$-module decomposition
\begin{equation}\label{Hodge}
\Wedge^*\mfp_+\otimes\mfg=\emph{im}(\partial)\oplus\ker(\square)\oplus\emph{im}(\partial^*)=\emph{im}(\partial)\oplus\ker(\partial^*)=\ker(\partial)\oplus\emph{im}(\partial^*)
\end{equation}
where $\square=\partial\partial^*+\partial^*\partial$.
\item[(3)]
In particular,  as
a $P_0$-module, we may identify $H_*(\mfp_+, \mfg)$ via \eqref{Hodge} with a submodule of $\Wedge^*\mfp_+\otimes\mfg.$ \end{itemize}
\end{prop}

\begin{notcon}\label{n.grade} Both $\partial^*$ and $\partial$  are compatible with the natural gradings on the spaces $\Wedge^\ell\mfp_+\otimes\mfg$ induced by \eqref{grading}.  Hence,  the homology groups
$H_\ell(\mfp_+, \mfg)$ are graded. We write $H_{\ell,r}(\mfp_+, \mfg)$ for the $r$-th grading component of $H_\ell(\mfp_+, \mfg)$, sitting inside
the $r$-th grading component of $\Wedge^\ell\mfp_+\otimes\mfg$ under the identification of $H_{\ell}(\mfp_+, \mfg)$ with a subspace of $\Wedge^\ell\mfp_+\otimes\mfg$ via Proposition \ref{p.Kostant}. Note that our convention for the gradation differs from the one used by Yamaguchi in \cite{Yamaguchi2} by a shift of $+1$.
\end{notcon}

\begin{defin}\label{d.symcon}
 It is convenient to introduce the following two  distinguished classes of the pairs $(\mfg, \Sigma)$ in Proposition \ref{p.rootparabolic}.  \begin{itemize}
\item[(1)] We say that $(\mfg, \Sigma)$ is
of \emph{symmetric type} if $\Sigma$ is a single long root and has depth $1$.
\item[(2)] We say that $(\mfg, \Sigma)$ is
of \emph{contact type} if  $\Sigma$ is a single long root and has depth   $2$  with  $\dim(\mfp_{\pm 2})=1$. \end{itemize}
The choices of $(\mfg,\alpha)$ corresponding to these cases are listed in Tables 1 and 3 of Section \ref{Appendix}.
\end{defin}

\begin{rem}
The terminology in Definition \ref{d.symcon} reflects the fact that the homogeneous space
corresponding to $(\mfg, \Sigma)$ of symmetric type (resp. contact type) is a Hermitian symmetric space (resp. a homogeneous contact manifold).
\end{rem}

The following results are deduced from \cite[Theorem 5.14]{Kostant}.

\begin{prop} \label{p.vanishing}
Let $(\mfg, \alpha_i)$ be a simple Lie algebra with a long simple root of $\mfg$ and let $\mfp$ be the associated parabolic subalgebra.
Identify
$\bigoplus_{r\geq 1}H_{2,r}(\mfp_+, \mfg)$ with a $\mfp_0$-submodule of $\Wedge^2\mfp_+\otimes\mfg$ via Proposition \ref{p.Kostant}.
If $H_{2,r}(\mfp_+, \mfg)\neq 0$ for $r \geq 1$, then any lowest weight space generating an irreducible component of the $\mfp_0$-module $H_{2,r}(\mfp_+, \mfg)$
is of the form
\begin{equation}\label{irred_comp}
\mfg_{\alpha_i}\wedge\mfg_{\alpha_i+\alpha_j}\otimes \mfg_{\beta_{ij}}\in \Wedge^2\mfp_1\otimes \mfp_{r-2},
\end{equation}
where $\alpha_j$ is a simple root connected to $\alpha_i$ in the Dynkin diagram of $\mfg$ and $\beta_{ij}$ is a lowest weight of the $\mfp_0$-module $\mfp_{r-2}$ (which is unique if $\mfp_{r-2}$ is irreducible).
Furthermore, the following holds.
\begin{enumerate}
\item[(1)] $H_{2,r}(\mfp_+, \mfg)=0$ for $r\geq 1$, unless $(\mfg, \alpha_i)$ is of symmetric or contact type, or equal to $(B_n/D_{n+1}, \alpha_3)$ for $n\geq 4$.
\item[(2)] If $(\mfg, \alpha_i)$ is of type of any of the exceptions in $(1)$, but is not equal to $(A_n, \alpha_1 \mbox{ nor }\alpha_{n})$ or to  $(A_n, \alpha_2 \mbox{ nor }\alpha_{n-1})$ for $n\geq 4$ or to $(B_n/D_n, \alpha_1)$, 
then $H_{2,r}(\mfp_+, \mfg)=0$ for $r\geq 2$ and $H_{2,1}(\mfp_+, \mfg)\neq 0$.
\item[(3)] If $(\mfg, \alpha_i)$ equals $(B_n/D_{n}, \alpha_1)$ or $(A_n\alpha_1 \mbox{ or }\alpha_{n} )$ for $n\geq 3$, then $H_{2}(\mfp_+, \mfg)=H_{2,2}(\mfp_+, \mfg)$, which is irreducible, except for $(D_{3}, \alpha_1)$, in which case it has two irreducible components.
\item[(4)] If
$(\mfg,\alpha_i)$ equals $(B_2, \alpha_1)$ or $(A_2,\alpha_1\mbox{ or } \alpha_2)$, then $H_{2}(\mfp_+, \mfg)=H_{2,3}(\mfp_+, \mfg)$ is irreducible. 
\item[(5)] If $(\mfg,\alpha_i)=(A_n, \alpha_2)$ for $n\geq 4$, then $H_{2}(\mfp_+, \mfg)=H_{2,1}(\mfp_+, \mfg)\oplus H_{2,2}(\mfp_+, \mfg)$.
\end{enumerate}
\end{prop}

\begin{proof} Recall that Theorem 5.14 of \cite{Kostant} (see \cite[Theorem 3.3.5]{csbook} for a formulation in the notation we use here) says that the
the irreducible components of $H_2(\mfp_+,\mfg)\subset \Wedge^2\mfp_+\otimes\mfg$ are in bijection with elements in the set $W^\mfp(2)$ consisting of the elements
of length $2$ in the Hasse diagram $W^\mfp$ of $\mfp$ (see \cite[Def. 3.2.1]{csbook}). We write $s_k$ for the root reflection corresponding to a simple root $\alpha_k$.
For a maximal parabolic subalgebra corresponding to a simple root $\alpha_i$ the algorithm in Section 3.2.14 of \cite{csbook} gives
$$W^\mfp(2)=\{s_is_j: \alpha_j \textrm{ is connected to } \alpha_i \textrm{ in the Dynkin diagram of } \mfg\}.$$ Moreover, by \cite[Proposition 3.2.16,(2)]{csbook}, for any $s_is_j\in W^\mfp(2)$
the subset of positive roots $\Phi_{s_is_j}$ appearing in \cite[Theorem 3.3.5]{csbook} (for a definition see also Section 3.2.14, p. 321, of \cite{csbook}) is given by $\Phi_{s_is_j} =\{\alpha_i, s_i(\alpha_j)\}$. Hence
$\Phi_{s_is_j}=\{\alpha_i, \alpha_i+\alpha_j\}$  if $\alpha_i$ is a long root.  For a maximal parabolic subalgebra corresponding to a long simple root $\alpha_i$,
Theorem 3.3.5 of \cite{csbook}  
 implies that the irreducible components of $H_{2}(\mfp_+, \mfg)$ are generated by the lowest weight spaces
\begin{equation}\label{irred_comp1}
\mfg_{\alpha_i}\wedge\mfg_{\alpha_{i}+\alpha_j}\otimes \mfg_ {-s_is_j(\theta)}\in \Wedge^{2}\mfp_1\otimes\mfg, \quad\textrm{for }\quad s_is_j\in W^\mfp(2),
\end{equation}
where $\theta$ is the highest root of $\mfg$, which proves the first claim.
For (1), note that Yamaguchi's list in \cite[Proposition 6.2]{Yamaguchi2} shows that any of the subspaces in \eqref{irred_comp1}
lies in non-positive homogeneities (that is, $\mfg_ {-s_is_j(\theta)}\in\mfg_{r-2}$ for $r\leq 2$) except for the cases mentioned in (1). Statements (2)-(5) can be also directly read off
from \cite[Proposition 6.2]{Yamaguchi2} or verified directly from \eqref{irred_comp1}.
\end{proof}

\subsection{Cone structures on rational homogeneous spaces}\label{s.homogenous}
Suppose $\mfg$ is a semisimple and  $\mfp=\mfp_\Sigma \subset \mfg$ a parabolic subalgebra. Let $P$ be the corresponding parabolic subgroup $\textrm{Aut}_{\textrm{fil}}(\mfg)$ of $\textrm{Aut}(\mfg)$ preserving the filtration \eqref{filtration} and set $G=\textrm{Aut}_0(\mfg)P$, where $\textrm{Aut}_0(\mfg)$ denotes the adjoint group of $\mfg$. We denote by $P_0$ and $P_+$ the Levi subgroup and the unipotent radical of $P$ respectively.

\begin{rem}
Note that the homogeneous space $G/P$ equals the quotient of the adjoint group of $\mfg$ by the connected parabolic subgroup of $\textrm{Aut}_0(\mfg)$ corresponding to $\mfp$.
\end{rem}
There is a natural isomorphism $TG/P\cong G\times_P\mfg/\mfp$ (via the left Maurer--Cartan form) and the filtration \eqref{filtration} induces a tangential filtration
$$T^{-i}G/P\cong G\times_P \mfp^{-i}/\mfp,\quad\quad i>0,$$ on the rational homogenous space $G/P$. The following is elementary.

\begin{lem}\label{l.G/Pfiltered}
In the Notation \ref{n.G/P}, since $P_+$ acts trivially on $\mfp^{-i}/\mfp^{-i+1}$ and $\mfp^{-i}/\mfp^{-i+1}\cong \mfp_{-i}$ as $P_0$-modules, one
has $\emph{gr}_{-i}(TG/P)\cong G/P_+\times _{P_0}\mfp_{-i}$, where the natural projection $G/P_+\rightarrow G/P$ is viewed as a $P_0$-principal bundle.
The filtration $\{T^{-i}G/P\}$ gives $G/P$ the structure of a filtered manifold with symbol algebra isomorphic to $\mfp_-:=\mfp_{-k}\oplus...\oplus\mfp_{-1}$.
 In particular,
 the distribution $$T^{-1}G/P\cong G\times_P\mfp^{-1}/\mfp\cong G/P_+\times_{P_0}\mfp_{-1} \subset TG/P$$ is bracket generating and $\{T^{-i}G/P\}$ is its weak derived flag.
\end{lem}

\begin{notcon}\label{n.mcCo}
Assume that $\Sigma$ consists of a single simple root $\alpha$ of a simple Lie algebra $\mfg$.
Then $\mfp^{-1}/\mfp$ is an irreducible $\mfp$-module and
 $\mfp^{-1}/\mfp\cong\mfp_{-1}$ has highest weight $-\alpha$  as an irreducible $\mfp_0$-module by Proposition \ref{p.rootparabolic}.
The space of highest weight vectors $\mfg_{-\alpha}\subset\mfp_{-1}$ determines a point
$z_0\in\mathbb P(\mfp_{-1})=\mathbb P(\mfp^{-1}/\mfp)$. We denote the $P_0$-orbit through $z_0$ by
$$\mcC_o^{G/P}:=P\cdot z_0=P_0\cdot z_0=P_0^{ss}\cdot z_0\subset \mathbb P(\mfp^{-1}/\mfp)=\mathbb P(\mfp_{-1}),$$
where $P_0^{ss}$ denotes the semisimple part of $P_0$.
\end{notcon}

Basic theory on parabolic subgroups and rational homogeneous spaces implies:

\begin{lem}\label{lem_highest_weight_cone}
In Notation \ref{n.mcCo}, the stabilizer in $P$ of $z_0$ is again a parabolic subgroup $Q$ of $G$ and the stabilizer $P_0^{ss}\cap Q$ in $P_0^{ss}$ of $z_0$ is a parabolic subgroup
of $P_0^{ss}$. Hence, $$\mcC_o^{G/P}\cong P/Q=P_0/P_0\cap Q=P_0^{ss}/P_0^{ss}\cap Q$$ is again a rational homogeneous space. The parabolic subalgebra $\mfq\leq\mfp\leq\mfg$ corresponding to $Q$ is associated to the
subset of simple roots consisting of $\alpha$ and the roots that are connected to $\alpha$ in the Dynkin diagram of $\mfg$.
\end{lem}

\begin{defin}\label{highest_weight_cone_struc}
In the setting of Notation \ref{n.mcCo}, the cone of highest weight vectors $\mcC_o^{G/P}\subset \mathbb P(\mfp^{-1}/\mfp)$ of the irreducible $P$-module $\mfp^{-1}/\mfp$ induces a $\mcC_o^{G/P}$--isotrivial cone structure on $G/P$ given by $$\mcC^{G/P}:=G\times _P\mcC^{G/P}_0\cong G/P_+\times_{P_0}\mcC^{G/P}_0\subset \mathbb P(T^{-1}G/P)\subset \mathbb P(TG/P)$$ and
we denote by $p: \mcC^{G/P}\rightarrow G/P$ the natural projection.
Note that, by Lemma \ref{lem_highest_weight_cone}, $\mcC^{G/P}\cong G\times _P P/Q\cong G/Q$ and hence, under this isomorphism, the tangent map $Tp$ of $p$ is given by the natural projection
\begin{equation}\label{tangent_map_of_cone_projection}
TG/Q\cong G\times_Q \mfg/\mfq\rightarrow G\times_P\mfg/\mfp\cong TG/P. 
\end{equation}
\end{defin}

To analyze the projection \eqref{tangent_map_of_cone_projection}, we compare the two filtrations \eqref{filtration} (and gradings) on $\mfg$ induced by $\mfp$ and $\mfq$ respectively. The following lemma, which holds for any nested pair of parabolic subalgebras $\mfq< \mfp< \mfg$ in a semisimple Lie algebra, follows directly from Proposition \ref{p.rootparabolic}:

\begin{lem}\label{nested_parabolics}
Suppose $\mfq<\mfp< \mfg$ be a nested pair of parabolic subalgebras of a semisimple Lie algebra $\mfg.$ The two gradings of $\mfg$ corresponding to $\mfp$ and $\mfq$,
$$\mfg= \oplus_{i \in \Z} \mfp_{i} = \oplus_{i \in \Z} \mfq_j,$$
satisfy the following.
\begin{enumerate}
\item[(1)] $\mfp_\pm\lhd\mfq_{\pm}$, $\mfq_0<\mfp_0$ and their centers satisfy $\mathfrak z(\mfp_0)<\mathfrak z(\mfq_0)$.
\item[(2)]
As a vector space $\mfg$ decomposes into the following direct sum of subalgebras of $\mfg$:
\begin{equation*}
\mfg=\mfp_-\oplus (\mfp_0\cap\mfq_-) \oplus\mfq_0\oplus(\mfp_0\cap\mfq_+)\oplus\mfp_+,
\end{equation*}
where $\mfp_0=(\mfp_0\cap\mfq_-)\oplus\mfq_0\oplus(\mfp_0\cap\mfq_+)$ and $\mfq_\pm=(\mfp_0\cap\mfq_\pm)\oplus\mfp_\pm$. In particular, the parabolic subalgebra $\mfp_0\cap\mfq$ of the reductive Lie algebra $\mfp_0$
equals $\mfq_0\oplus(\mfp_0\cap\mfq_+)$, where $\mfq_0$ is its Levi subalgebra and $\mfp_0\cap\mfq_+$ its nilradical.

\end{enumerate}
\end{lem}

\begin{prop}\label{long_versus_short_root}
Suppose $\mfg$ is a simple Lie algebra and $\mfp$ is a maximal parabolic subalgebra corresponding to a long simple root $\alpha$. Let $P\leq G$ be as above and let $Q$ be as in Lemma \ref{lem_highest_weight_cone}.
Consider the cone structure $\mcC^{G/P}\cong G/Q\subset \mathbb P (TG/P)$ on $G/P$ as defined in Definition \ref{highest_weight_cone_struc}.
Then one has:
\begin{enumerate}
\item[(1)] $\mcC_o^{G/P}\cong P_0/P_0\cap Q$ is a compact Hermitian symmetric space (equivalently, $\mfp_0\cap\mfq_\pm<\mfp_0$ is an abelian subalgebra).
\item[(2)] The subbundles $T^{-i}G/Q\cong G\times_Q \mfq^{-i}/\mfq\subset TG/Q$ coincide with $\mcD^{-i}$ as in Definition \ref{d.cone_filtration} for $i=1,2, 3,4$.

\item[(3)] The line subbundle
$\mcF\subset \mcD^{-1}$ defined by $\mfq_{-1}^F:=\mfg_{-\alpha} \subset\mfq_{-1}$ is a characteristic conic connection on $\mcC^{G/P}$ (that is, $[\mcF, \mcD^{-2}]\subset \mcD^{-2}$). Consequently,
$[\mcF, \mcD^{-3}]\subset \mcD^{-3}$ by Lemma \ref{l.-3}. 

\item[(4)] $\mcD^{-4}=(Tp)^{-1}(T^{-1}G/P)$.
\end{enumerate}
\end{prop}
\begin{proof}
The Dynkin diagram of the parabolic subalgebra $\mfp_0^{ss}\cap\mfq$ of $\mfp_0^{ss}$ arise from that of $\mfq$ by removing the cross corresponding to $\alpha$ and the edges emanating from $\alpha$. Hence, (1) can be checked from the classification of irreducible compact Hermitian symmetric spaces given in Table 1 of Section 5 which uses the convention in Notation \ref{n.Dynkin}.

Next, let us prove $T^{-1} G/Q = \mathcal{D}^{-1}$ in (2). Note that we have the vector space decomposition
$$\mathfrak{q}^{-1}/\mathfrak{q}=\mathfrak{q}_{-1}=\mathfrak{q}_{-1}\cap\mathfrak{p}_-\oplus\mathfrak{q}_{-1}\cap\mathfrak{p}_0.$$ Since $\mathfrak{q}$ corresponds to the subset of simple roots consisting of  $\alpha$ and any roots connected to $\alpha$ in the Dynkin diagram of $\mathfrak{g}$, we must have $\mathfrak{q}_{-1}\cap\mathfrak{p}_-=\mathfrak{q}_{-1}^F,$ defined as in (3). Hence, \eqref{tangent_map_of_cone_projection} shows that $T^{-1} G/Q\subset \mathcal{D}^{-1}$. To show the equality $T^{-1} G/Q = \mathcal{D}^{-1}$, it suffices to prove that $T^{-1} G/Q$ contains the vertical bundle $\mathcal{V}\cong G\times _Q \mathfrak{p}/\mathfrak{q}$ of $p$. This is the case if and only if $\mathfrak{p}/\mathfrak{q}=\mathfrak{p}_0\cap \mathfrak{q}_-$ equals $\mathfrak{p}_0\cap \mathfrak{q}_{-1}=:\mathfrak{q}_{-1}^V$. This in turn holds if and only if $\mathfrak{p}_0\cap \mathfrak{q}_-\leq \mathfrak{p}_0$ is abelian, which is satisfied by $(1)$. Hence, $T^{-1} G/Q=\mathcal{D}^{-1}=\mathcal{F}\oplus\mathcal{V}.$ As a byproduct, we have shown that $\mathcal{F}$ defines a conic connection on $\mathcal{C}^{G/P}$. By Proposition \ref{p.HM04}  and $T^{-1} G/Q=\mathcal{D}^{-1}$, we see that also $T^{-2}G/Q=\mcD^{-2}$, since  $T^{-2}G/Q$ is the first derived system of $T^{-1} G/Q$.

Next, we prove that $\mathcal{F}$ is a characteristic conic connection. We have to show that the
Lie bracket $\mathfrak{q}_{-1}^F\otimes\mathfrak{q}_{-2}\rightarrow \mathfrak{q}_{-3}$ is zero. In other words, we need to show that $\beta -\alpha$ is not a root  for any root $\beta$ corresponding to $\mathfrak{q}_{-2}$.
Since $\mathfrak{q}_{-1}^F\otimes\mathfrak{q}_{-1}^V\cong\mathfrak{q}_{-2}$ via the Lie bracket from Corollary \ref{c.isom}, a root $\beta$ corresponding to $\mathfrak{q}_{-2}$ satisfies
\begin{itemize}
\item[(i)] $\textrm{ht}_\alpha(\beta)=\textrm{ht}_{\Sigma_\mathfrak{q}\setminus\{\alpha\}}(\beta)=-1$;
\item[(ii)] $\beta+\alpha$ is a root corresponding to $\mathfrak{q}_{-1}^V$; and
\item[(iii)] $\beta+k\alpha$ is not a root for $k>1$.
\end{itemize}
It follows that $\frac{2\langle \alpha, \beta\rangle}{\langle \alpha, \alpha\rangle}=-2+1=-1$, because $\alpha$ is a long root. This shows that $\beta-\alpha$ is not a root and $\mathcal{F}$ is a characteristic conic connection.

Next, we check (2) for $i=3, 4$. We have
$$T^{-3}G/Q=[T^{-2}G/Q, T^{-1}G/Q]+T^{-2}G/Q=[\mcD^{-2}, \mcV]+\mcD^{-2}=\mcD^{-3},$$ where the second equality follows from $\mcF$ being characteristic and $\mcD^{-2}=T^{-2}G/Q$.
By Lemma \ref{l.-3}, we also have $T^{-4}G/Q=\mcD^{-4}$.

To prove (4), the inclusion
 $T^{-4}G/Q\subset (Tp)^{-1}(T^{-1} G/P)$ is immediate from $T^{-4}G/Q=\mcD^{-4}$. To check that this inclusion is an equality, note that any root $\beta$ corresponding to $\mathfrak{p}_{-1}$ can be written as
$$\beta=-\alpha_i-\sum_{i \neq j}n_j\alpha_j$$ for some integers $n_j\geq 0$ and $\alpha=\alpha_i$. Since $\alpha_i$ is a long root, $$2\frac{\langle\beta, \alpha_i\rangle}{\langle\alpha_i, \alpha_i\rangle}=-2+n_{i-1}+n_{i+1}$$
is an integer between $-2$ and $1$. Hence, $0\leq n_{i-1}+n_{i+1}\leq 3$, which shows that $(Tp)^{-1}(T^{-1} G/P)\subset T^{-4}G/Q$.
\end{proof}

\begin{rem} When $\alpha$ is a short simple root,
we can still define the conic connection $\mcF$ of Proposition \ref{long_versus_short_root}. It turns out, however, that $\mcF$ is not a characteristic connection in this case.  In fact, one can check using Proposition \ref{p.Hw12} that $\mcC^{G/P}$ does not admit a characteristic conic connection. \end{rem}

Since we are interested in cone structures admitting characteristic connections, we fix for later the following notation.

\begin{notcon}\label{n.pq} Suppose that $(\mfg, \alpha)$ is a simple Lie algebra with a choice of long simple root. Let $\mfq\leq \mfp\leq \mfg$ be the parabolic subalgebras defined in
Lemma \ref{lem_highest_weight_cone}. Then we set
$$\mfq_{\pm 1}^F:=\mfg_{\pm \alpha}=\mfp_{\pm}\cap\mfq_{\pm 1} \quad\textrm{ and }\quad \mfq_{\pm 1}^V:=\mfp_0\cap\mfq_{\pm 1}=\mfp_0\cap\mfq_{\pm}$$ such that $\mfq_{\pm 1}=\mfq_{\pm 1}^F\oplus\mfq_{\pm1}^V$, which
gives rise to the decomposition $\mcD^{-1}=\mcF\oplus\mcV$.
\end{notcon}

An immediate consequence of Proposition \ref{long_versus_short_root} is:

\begin{cor}\label{FF_model}
Assume the setting of Proposition \ref{long_versus_short_root}. Then the Lie bracket $\mfq_{\pm 1}^F\otimes \mfq_{\pm 1}^V\rightarrow \mfq_{\pm 2}$ defines an isomorphism. 
\end{cor}

Recall that, if $\mfp_\Sigma$ is any parabolic subalgebra of a simple Lie algebra $\mfg$, then the depth of $\Sigma$ equals $\textrm{ht}_{\Sigma}(\theta)$, where $\theta$ is the highest root of $\mfg$. Hence,
Lemma \ref{lem_highest_weight_cone} and Tables 1--3 of Section \ref{Appendix} imply the following.

\begin{prop}\label{p.depth}
Suppose $(\mfg,\alpha)$ a simple Lie algebra with a long simple root with corresponding parabolic subalgebra $\mfp$ and let $\mfq$ be as in Lemma \ref{lem_highest_weight_cone}.
\begin{itemize} \item[(1)]
 If $(\mfg, \alpha)\neq (A_n, \alpha_1/\alpha_n)$ is of symmetric type, then the depth of $\mfq$ is 3. If $\mfg=(A_n, \alpha_1/\alpha_n)$, then it is 2. 
 \item[(2)] If $(\mfg, \alpha)$
is of contact type, then the depth of $\mfq$ is 5. \item[(3)]
If $(\mfg, \alpha)$ is $(B_n/D_{n+1}, \alpha_3)$ for $n\geq 4,$ then the depth of $\mfp$ is 2, while the depth of  $\mfq$ is 6. \end{itemize}
In cases (2) and (3),  Proposition \ref{long_versus_short_root} implies $\mfp^{\pm 1}=\mfq^{\pm 4}$.
\end{prop}

\subsection{Nested pair of parabolic subalgebras and their homology groups}
Suppose $(\mfg,\alpha)$ is a simple Lie algebra with a choice of long simple root and let $\mfq\leq \mfp\leq \mfg$ be the parabolic subalgebras defined in
Lemma \ref{lem_highest_weight_cone}. Let $Q < P< G$ as in Section \ref{s.homogenous}. Denote by $P_+<Q_+$ the unipotent radicals and by $Q_0<P_0$ the Levi subgroups of $Q<P$. The following lemma is immediate (cf. Proposition 2.3 of \cite{CapCores}).

\begin{lem}\label{l.rel.hom}
Let $\iota: \Wedge^* (\mfg/\mfp)^*\otimes\mfg\cong \Wedge^* \mfp_+\otimes\mfg\hookrightarrow \Wedge^* (\mfg/\mfq)^*\otimes\mfg\cong \Wedge^* \mfq_+\otimes\mfg$ be the natural inclusion. Then 
$\iota$ intertwines the boundary operators for the homology groups $H_*(\mfp_+,\mfg)$ and $H_*(\mfq_+,\mfg)$. 
\end{lem}

From Lemma \ref{l.rel.hom} one sees that the complete reducible $Q$-module $H_2(\mfq_+, \mfg)$ must contain a component isomorphic to a completely reducible quotient of the $Q$-module 
$H_2(\mfp_+, \mfg)$. 
\begin{prop}\label{p.rel.Lie} 
In Notation \ref{n.pq} we have:
\begin{enumerate}
\item[(1)] The largest completely reducible quotient of 
$H_2(\mfp_+, \mfg)$ as a $Q$-module equals 
$$H_2(\mfp_+, \mfg)/\mfq_+ H_2(\mfp_+, \mfg)=H_2(\mfp_+, \mfg)/\mfq_1^VH_2(\mfp_+, \mfg)=H_0(\mfq_{1}^V,  H_2(\mfp_+, \mfg)).$$
and is contained in $H_2(\mfq_+,\mfg)$.
\item[(2)] Under the identification \eqref{Hodge}
of $H_2(\mfq_+, \mfg)$ with a $Q_0$-submodule of $\Wedge^2\mfq_+\otimes \mfg$ one has
\begin{equation}\label{rel_Lie_2}
H_0(\mfq_{1}^V, H_{2}(\mfp_+,\mfg))\subset \Wedge^{2}\mfp_+\otimes \mfg\subset \Wedge^2\mfq_+\otimes\mfg
\end{equation}
and the irreducible components of the $\mfq_0$-module $H_0(\mfq_{1}^V, H_{2}(\mfp_+,\mfg))\subset \Wedge^{2}\mfp_+\otimes \mfg$ are
generated by the lowest weight vectors of the irreducible components of the $\mfp_0$-module  $H_{2}(\mfp_+,\mfg)$. 
\end{enumerate}
\end{prop}
\begin{proof}
On a completely reducible quotient of the $Q$-module $H_{2}(\mfp_+, \mfg)$ the unipotent radical $Q_+$ of $Q$ must act trivially and hence the largest such quotient equals $H_2(\mfp_+, \mfg)/\mfq_+ H_2(\mfp_+, \mfg)$.
Since $\mfp_+$ acts trivially on $H_2(\mfp_+, \mfg)$, that quotient equals $H_2(\mfp_+, \mfg)/\mfq_1^VH_2(\mfp_+, \mfg)$, which by definition is the $0$-th Lie algebra homology group of the 
abelian Lie algebra $\mfq_{1}^V$ with values in the representation $H_2(\mfp_+, \mfg)$. Recall that $\mfq_0\oplus\mfq_{1}^V$ is a parabolic subalgebra of the reductive Lie algebra 
$\mfp_0$ with nilradical $\mfq_{1}^V$. Hence, $H_0(\mfq_{1}^V,  H_2(\mfp_+, \mfg))$ and $H_2(\mfq_+, \mfg)$ are both subject to Kostant's Theorem \cite[Theorem 5.14]{Kostant}, 
which implies that  $H_0(\mfq_{1}^V,  H_2(\mfp_+, \mfg))\subset H_2(\mfq_+, \mfg)$ and also the statements in (2) (cf. also \cite[Theorem 2.7]{CapSoucek}).
\end{proof}

In the following proposition, the subspace $H_0(\mfq_{1}^V, H_2(\mfp_+, \mfg))\subset H_2(\mfq_+, \mfg)$ is identified with a $Q_0$-submodule of $\Wedge^2\mfp_+\otimes\mfg\subset \Wedge^2\mfq_+\otimes \mfg$ according to \eqref{rel_Lie_2}.

\begin{prop}\label{harm_curv_C}
In Notation \ref{n.pq} the following holds.
\begin{enumerate}
\item[(1)] Assume $(\mfg,\alpha)$ is of symmetric type.
\begin{itemize}
\item[(i)] Assume $(\mfg,\alpha)\neq (A_n,\alpha_1 \mbox{ or } \alpha_n)$.  Then $H_0(\mfq_1^V, H_{2,1}(\mfp_+, \mfg))\subset \mfq_{1}^F\otimes \mfq_{2}\otimes\mfq_{-3}$ and the Hodge decomposition \eqref{Hodge} of $\Wedge^2\mfq_+\otimes\mfg$ restricted to
$\mfq_{1}^F\otimes \mfq_{2}\otimes\mfq_{-3}$  reads as
\begin{equation}\label{Hodge1}
\mfq_{1}^F\otimes \mfq_{2}\otimes\mfq_{-3}=\ker(\partial^*)\oplus \emph{im}(\partial)=H_0(\mfq_1^V, H_{2,1}(\mfp_+, \mfg))\oplus\emph{im}(\partial).
\end{equation}
Moreover, under the identification $\mfq_{1}^F\otimes \mfq_{-2}^*\otimes\mfq_{-3}\cong \mfq_{1}^F\otimes \mfq_{2}\otimes\mfq_{-3}$ via the Killing form, $\emph{im}(\partial)$ coincides with the image of the inclusion
$$\emph{Id}\otimes \emph{ad}(_-)\vert _{\mfq_{-2}}: \mfq_{1}^F\otimes\mfq_{-1}^V\hookrightarrow \mfq_{1}^F\otimes \mfq_{-2}^*\otimes\mfq_{-3}.$$
\item[(ii)] $H_0(\mfq_1^V, H_{2,2}(\mfp_+, \mfg))\subset \mfq_1^F\otimes\mfq_2\otimes\mfq_{-1}^V$. Moreover, the isomorphism $\mfq_{\pm 1}^F\otimes\mfq_{\pm 1}^V\cong \mfq_{\pm 2}$ of Corollary \ref{FF_model}
and the Killing form induce
isomorphisms
\begin{equation}\label{isos}
\mfq_1^F\otimes\mfq_2\otimes\mfq_{-1}^V\cong S^2\mfq_1^F\otimes \emph{End}(\mfq_{-1}^V)\cong S^3\mfq_{1}^F\otimes \emph{Hom}(\mfq_{-1}^V, \mfq_{-2}).
\end{equation}
Under these isomorphisms, the subspace $\ker(\partial^*)$  corresponds to
\begin{align*}
\ker(\partial^*)&=H_0(\mfq_1^V, H_{2,2}(\mfp_+, \mfg))\oplus\emph{im}(\partial^*)\\
&\cong S^2\mfq_1^F\otimes \emph{End}_0(\mfq_{-1}^V)\cong S^3\mfq_{1}^F\otimes \emph{Hom}_0(\mfq_{-1}^V, \mfq_{-2}),
\end{align*}
where $\emph{End}_0(\mfq_{-1}^V)\subset \emph{End}(\mfq_{-1}^V) $ and $\emph{Hom}_0(\mfq_{-1}^V, \mfq_{-2})\subset \emph{Hom}(\mfq_{-1}^V, \mfq_{-2})$  denote the subspaces of trace-free maps.
\end{itemize}
\item[(2)] Assume $(\mfg,\alpha)$ is of contact type or equal to $(B_n/D_{n+1}, \alpha_3)$ for $n\geq 4$, then one has $H_0(\mfq_1^V, \oplus_{r\geq 1}H_{2,r}(\mfp_+, \mfg))=H_0(\mfq_1^V, H_{2,1}(\mfp_+, \mfg))$ and
the Hodge decomposition \eqref{Hodge} of $\Wedge^2\mfq_+\otimes\mfg$ restricted to $\mfq_1^F\otimes\mfq_2\otimes\mfq_{-4}$ reads as
$$\mfq_1^F\otimes\mfq_2\otimes \mfq_{-4}=\ker(\partial^*)=H_0(\mfq_1^V, H_{2,1}(\mfp_+, \mfg))\oplus\emph{im}(\partial^*).$$
\end{enumerate}
\end{prop}
\begin{proof}
By (2) of Proposition \ref{p.rel.Lie},  the description of the lowest weight vectors of the $\mfp_0$-module $H_{2,r}(\mfp_+, \mfg)$ in \eqref{irred_comp} and the description of $\mfq$ in terms of roots in Lemma
\ref{lem_highest_weight_cone} imply that:
\begin{itemize}
\item $H_0(\mfq_{1}^V,  H_{2,1}(\mfp_+, \mfg))$ is contained in $\mfq_1^F\otimes\mfq_2\otimes\mfq_{-3}$ in the symmetric case for  $(\mfg,\alpha)\neq (A_n,\alpha_1 \mbox{ or } \alpha_n)$ and in
$\mfq_1^F\otimes\mfq_2\otimes\mfq_{-4}$ in the cases in (2), since any lowest weight vector of the irreducible $\mfp_0$-module $\mfp_{-1}$ lies in $\mfq_{-3}$ in the first case and in $\mfq_{-4}$ in the latter cases;
\item $H_0(\mfq_{1}^V,  H_{2, 2}(\mfp_+, \mfg))$ is contained in $\mfq_1^F\otimes\mfq_2\otimes\mfq_{-1}^V$, since any lowest weight vector of the adjoint representation of
$\mfp_0=\mfq_{-1}^V\oplus\mfq_0\oplus\mfq_{1}^V$ lies in $\mfq_{-1}^V$;
\end{itemize}
Next let us prove the remaining statements in (i). Given \eqref{Hodge}, the decomposition \eqref{Hodge1}
is equivalent to $\textrm{im}(\partial^*)\cap \mfq_{1}^F\otimes \mfq_{2}\otimes\mfq_{-3}=0$. Since $\partial^*$ preserves the grading on $\Wedge^*\mfq_+\otimes\mfg$, the formula of $\partial^*$ implies that only the image of the restriction of $\partial^*$ to 
$$\Wedge^2\mfq_{1}^V\otimes\mfq_{1}^F\otimes \mfq_{-3}\subset\Wedge^3\mfq_1\otimes \mfq_{-3}\subset \Wedge^3\mfq_+\otimes\mfg$$
can possibly have a non-zero intersection with $\mfq_{1}^F\otimes \mfq_{2}\otimes\mfq_{-3}$. Since $[\mfq_{1}^V,\mfq_1^V]=0$, this is however not the case,
which proves \eqref{Hodge1}. For the last claim in (i), recall that $\partial$ also preserves the grading on $\Wedge^*\mfq_+\otimes\mfg$ and so the definition of $\partial$
implies that the only component in $\mfq_+\otimes\mfg$ that is mapped to $\mfq_{1}^F\otimes\mfq_2\otimes \mfq_{-3}$ is the component $\mfq_1^F\otimes\mfq_{-1}^V$ of degree $0$.
Under the isomorphism $\mfq_1^F\otimes\mfq_{-1}^V\cong \textrm{Hom}(\mfq_{-1}^F, \mfq_{-1}^V)$ via the Killing form, we have
$$\partial \phi(X,Y)=-\textrm{ad}(\phi(X))(Y)\quad\textrm{ for } X\in\mfq_{-1}^F, Y\in\mfq_{-2} \textrm{ and }  \phi\in \textrm{Hom}(\mfq_{-1}^F,\mfq_{-1}^V),$$
since $[\mfq_{-1}^F,\mfq_{-2}]=0$ by (3) of Proposition \ref{long_versus_short_root}. This proves the last claim of (i), where injectivity follows from Lemma \ref{FF.injective} and Table 1 of Section \ref{Appendix}.

Now we prove the remaining claims in (ii). The statement about the isomorphisms in \eqref{isos} is clear.
For the last statement of (ii), note that $\partial^*$ restricted to $\mfq_1^F\otimes\mfq_2\otimes\mfq_{-1}^V$ is given by
\begin{align*}
&\partial^*: \mfq_1^F\otimes\mfq_2\otimes\mfq_{-1}^V\rightarrow \mfq_1^F\otimes\mfq_1^F \\
&\partial^*(X\otimes Y\otimes V)=-X\otimes [Y, V],
\end{align*}
since $[\mfq_{1}^F, \mfq_{-1}^V]=0$ and $[\mfq_{1}^F, \mfq_{2}]=0$. For $Y\in\mfq_2$ and $V\in\mfq_{-1}^V$, the bracket $[Y,V]$ is nonzero if and only if $Y$ corresponds
under the isomorphism $\mfq_{1}^F\otimes\mfq_{1}^V\cong \mfq_2$ to an element of the form $Z\otimes W\in \mfq_{1}^F\otimes\mfq_{1}^V$, where $W\in\mfq_{1}^V$ is dual to $V$ under the Killing form.
Hence, under the first isomorphism in \eqref{isos}, the operator $\partial^*$ is just given by the natural contraction $S^2\mfq_1^F\otimes \textrm{End}(\mfq_{-1}^V)\rightarrow S^2\mfq_1^F$  (in the endomorphism factor),
which finishes the proof of (ii).

Now let us prove (2). We already observed that  $H_0(\mfq_1^V, H_{2,1}(\mfp_+, \mfg))\subset \mfq_1^F\otimes\mfq_2\otimes\mfq_{-4}$ and by
Proposition \ref{p.vanishing}, we have $H_0(\mfq_1^V, \oplus_{r\geq 1} H_{2,r}(\mfp_+, \mfg))=H_0(\mfq_1^V, H_{2,1}(\mfp_+, \mfg))$. So it remains to show that
$\ker(\partial^*)=\mfq_1^F\otimes\mfq_2\otimes\mfq_{-4}$. Since $[\mfq_1^F, \mfq_{2}]=0,$ it suffices to show   \begin{equation}\label{e.243} [\mfq_2, \mfq_{-4}]=0 \mbox{ and } [\mfq_1^F, \mfq_{-4}]=0. \end{equation}
By (2) of Proposition \ref{long_versus_short_root}, one must have
$$[\mfq_2, \mfq_{-4}]\in\mfq_{-2}\cap[\mfp_{1}, \mfp_{-1}] \mbox{ and } [\mfq_1^F, \mfq_{-4}]\in \mfq_{-3}\cap[\mfp_{1}, \mfp_{-1}].$$ But $[\mfp_{1}, \mfp_{-1}]\subset\mfp_0$, while $\mfp_0\cap \mfq_-=\mfq_{-1}^V$  by (1) of Proposition \ref{long_versus_short_root}. This shows \eqref{e.243}, and the proof of (2) is completed.
\end{proof}

\section{Cone structures given by varieties of minimal rational tangents}\label{s.vmrt}

In this section, we give a brief survey of results on varieties of minimal rational tangents, which provide many examples of cone structures with characteristic conic connections arising from algebraic geometry.
Although these results have been the motivation to investigate the questions we study in the next section, they are logically independent.   As the nature of the results in this section and the methods used to prove them are quite different from those in the rest of the paper, we give a minimal presentation which we believe is sufficient for readers outside algebraic geometry to understand at least the statements of the three results,  Theorem \ref{t.Mori}, Proposition \ref{p.vmrt} and Theorem \ref{t.Mok}.

Firstly, recall the following facts on vector bundles on the Riemann sphere $\mbP^1$. We have the tautological line bundle $\mcO(-1)$ and its dual line bundle $\mcO(1)$ on $\mbP^1$.  Any holomorphic line bundle is of the form $\mcO(\ell)$ for an integer $\ell$ where $\mcO(\ell) :=\mcO(1)^{\otimes \ell}$ if $\ell \geq 0$ and $\mcO(\ell) := \mcO(-1)^{\otimes - \ell}$ if $\ell <0$.  The line bundle $\mcO(\ell)$ has no nonzero global holomorphic sections on $\mbP^1$ if $\ell <0$.  Moreover,  any holomorphic vector bundle on $\mbP^1$ is isomorphic to a direct sum of holomorphic line bundles.

  \begin{defin}\label{d.unbendable}
  The image of a nonconstant holomorphic map $f: \mbP^1 \to X$ from the Riemann sphere to a complex manifold is called a \emph{rational curve} on $X$. \begin{itemize} \item[(1)] Given a rational curve $C \subset X$, we can always choose $f$ such that it is injective on a dense open subset of $\mbP^1$. Such $f$ is called a \emph{ normalization } of $C$. \item[(2)]  The \emph{anti-canonical degree} of $C \subset X$ is the integer $\ell$ such that the line bundle $ \det (f^* TX)$ on $\mbP^1$ under a normalization $f$ is isomorphic to $\mcO(\ell)$. \item[(3)]   A rational curve $C \subset X$ is an \emph{immersed rational curve} on $X$ if its normalization is a holomorphic immersion.  An immersed rational curve $C \subset X$ is called an \emph{unbendable rational curve} if its normal bundle $N_{C} = f^*TX/TC$ is isomorphic to  $\mcO(1)^{\oplus r} \oplus \mcO^{\oplus (n-r)}$ as vector bundles on $\mbP^1$ for some nonnegative integer $r$ and $n= \dim X -1$. The number $r+2$ is  the anti-canonical degree of $C$.  \end{itemize} \end{defin}

The following theorem summarizes main results on minimal rational curves on uniruled projective manifolds, the key points of which are bend-and-break argument going back to Mori and the study of singularities of minimal rational curves due to Kebekus. Interested readers can consult \cite{HwangMSRI}, \cite{HwangICM} and  \cite{Hwang-Mok99} for more details and further references.

\begin{thm}\label{t.Mori}
Let $X$ be a projective manifold, i.e., a compact complex manifold that can be embedded in the complex projective space $\mbP^N$ for some $N >0$. Assume that $X$ is uniruled, which means that for any $x \in X$, there exists a rational curve $C \subset X$ containing $x$. Then there exist complex manifolds $\mcU$ and $\mcK$ with holomorphic maps $\rho: \mcU \to \mcK$ and $\mu: \mcU \to X$ with the following properties.
\begin{itemize}
\item[(1)] $\rho$ is a $\mbP^1$-bundle and there exists a closed analytic subset $E_{\rho}\subsetneq \mcK$ such that for any $z \in \mcK \setminus E_{\rho}$, the holomorphic map $ \mu|_{\rho^{-1}(z)}: \rho^{-1}(z) \cong \mbP^1 \to X$ is a normalization of an unbendable rational curve in $X$ the anti-canonical degree of which is $2 + \dim \mcU - \dim X$.
\item[(2)] There exists a closed analytic subset $E_{\mu} \subsetneq X$ such that for any $x \in X \setminus E_{\mu}$, the fiber $\mu^{-1}(x)$ is a compact complex manifold and at every point $y \in \mu^{-1}(x)$,
    $${\rm Ker}(T_y \mu) \cap {\rm Ker}(T_y \rho) =0. $$ \end{itemize}
Such a double-fibration $\mcK \stackrel{\rho}{\leftarrow} \mcU \stackrel{\mu}{\to} X$ is called a family of minimal rational curves on $X$. 
\end{thm}

\begin{exa}\label{e.lines}
Let $X$ be a closed complex submanifold of $\mbP^m$. We say that $X$ is covered by lines of $\mbP^m$ if for each point $x \in X$, there exists a line of $\mbP^m$ lying on $X$ passing through $x$. In this case, there exists a family of minimal rational curves on $X$ whose members are lines of $\mbP^m$.  There are two well-known examples of such $X$: any nonsingular hypersurface $X \subset \mbP^{n+2}$ of degree $<n+2$ and a minimal  $G$-equivariant embedding $G/P \subset \mbP^m$ of the homogeneous space of a complex simple Lie group $G$ modulo a maximal parabolic subgroup $P.$ \end{exa}

\begin{defin}\label{d.vmrt}
Given a family of minimal rational curves on $X$ as in Theorem \ref{t.Mori}, the condition ${\rm Ker}(T_y \mu) \cap {\rm Ker}(T_y \rho) =0$ in (2) implies that $T\mu ({\rm Ker}(T_y \rho))$ is a 1-dimensional subspace in $T_x X$ for any $x \in X \setminus E_{\mu}$ and $y \in \mu^{-1}(x).$
Define the tangent map $$\tau: \mcU \setminus \mu^{-1}(E_{\mu}) \to \mbP TX$$ by $\tau (y):= [T\mu ({\rm Ker}(T_y \rho))].$
The image $\tau (\mu^{-1}(x))$ for $x \in X\setminus E_{\mu}$ is called the \emph{variety of minimal rational  tangents} at $x$ of the family of
minimal rational curves. \end{defin}

\begin{prop}\label{p.vmrt}
In Definition \ref{d.vmrt}, assume that $\tau$ is an embedding such that the image of $\tau$ defines a cone structure $\mcC$ on $X \setminus E_{\mu}$. Then $T \tau ({\rm Ker}(T \rho))$ gives a characteristic conic connection on this cone structure, whose cubic torsion is identically zero. \end{prop}

\begin{proof}
That $\mcF:= T \tau ({\rm Ker}(T \rho))$ gives a characteristic conic connection is proved in Proposition 8 
of  \cite{Hwang-Mok04}. We give a slightly different proof here. By assumption, the line bundle $\mcF$ and the vector bundle $\mcV$  are defined on $\mu^{-1}(X \setminus E_{\mu})$.  Pick a point $z \in (\mcK \setminus E_{\rho}) \cap \rho(\mu^{-1}(X \setminus E_{\mu}))$ and let $C \subset X$ be the unbendable rational curve given by $\mu(\rho^{-1}(z)) \subset X$. Denote by $N_C^+$ the  $\mcO(1)^{\oplus r}$-component of $N_C$ in Definition \ref{d.unbendable}.  The line bundle $\mcF$ restricted to $\rho^{-1}(z) \cap  \mu^{-1}(M \setminus E_{\mu})$ can be extended to the tangent bundle $TC \cong \mcO(2)$.  The vertical bundle $\mcV$ restricted to $\rho^{-1}(z) \cap  \mu^{-1}(M \setminus E_{\mu})$ can be extended to the vector bundle ${\rm Hom}(TC, N_C^+) \cong \mcO(-1)^{\oplus r}$ on $\rho^{-1}(z)$.  By Corollary \ref{c.isom},  ${\rm gr}_{-2}(\mcD)$ restricted to  $\rho^{-1}(z) \cap  \mu^{-1}(M \setminus E_{\mu})$ can be extended to a vector bundle on $\rho^{-1}(z)$ isomorphic to $\mcO(1)^{\oplus r}$. On the other hand, the vector bundle $T\mcC/\mcD^{-2}$ restricted to $\rho^{-1}(z) \cap  \mu^{-1}(M \setminus E_{\mu})$ can be extended to a vector bundle on $\rho^{-1}(z)$ isomorphic to $\mcO^{\oplus (n-r)}$. 
It follows that the characteristic torsion of $\mcF$ restricted to $\rho^{-1}(z) \cap  \mu^{-1}(M \setminus E_{\mu})$  can be extended to a section of
$$\mcF^* \otimes {\rm gr}_{-2}(\mcD)^* \otimes T\mcC / \mcD^{-2}|_{\rho^{-1}(z)} \cong \mcO(-3)^{\oplus r(n-r)},$$ which cannot have nonzero holomorphic section. Thus the characteristic torsion is identically zero.
The cubic torsion restricted  to $\rho^{-1}(z) \cap  \mu^{-1}(M \setminus E_{\mu})$ can be extended to  a section of  $$S^3 \mcF^*\otimes \textrm{Hom}(\mcV, \mcV \otimes \mcF))|_{\rho^{-1}(z)} \cong \mcO(-4)^{\oplus r^2},$$ which cannot have nonzero holomorphic section. Thus the cubic torsion is identically zero.
\end{proof}

\begin{exa}
The fact that lines on $\mbP^m$ are determined by their tangent vector at one point implies that the family of minimal rational curves consisting of lines in Example \ref{e.lines} satisfy the assumption of Proposition \ref{p.vmrt}. When $X = G/P$ is a rational homogeneous space defined by $(\mfg, \alpha)$ for a long simple root $\alpha$, one can check (see \cite[Proposition 1]{Hwang-Mok02}) that the cone structure $\mcC^{G/P}=G \times_P \mcC^{G/P}_o$ is exactly the one given by varieties of minimal rational tangents as in Proposition \ref{p.vmrt}.
\end{exa}

The following result is a slightly modified version of Theorem \ref{t.longroot}. 
In fact, Theorem \ref{t.longroot} can be derived from it by the Cartan--Fubini type extension theorem \cite[Main Theorem]{HM01}.
As discussed in the introduction, the proof due to Mok \cite{Mok}  combines methods of algebraic geometry and parabolic geometry. In the next section, we give an alternative proof of a stronger result, which was one of the motivation of this paper.

\begin{thm}\label{t.Mok}
In Proposition \ref{p.vmrt}, assume that $\tau(\mu^{-1}(x)) \subset \mbP T_x X$ for a general $x \in X \setminus E_{\mu}$ is projectively isomorphic to $\mcC_o^{G/P} \subset \mbP T_o G/P$ for  $G/P$ associated to $(\mfg, \alpha)$ for a long simple root $\alpha$ and the distribution on $X$ determined by the linear span of the cone structure $\mcC \subset \BP T(X\setminus E_\mu)$ is bracket-generating. Then the cone structure in Proposition \ref{p.vmrt} is locally isomorphic to the cone structure $\mcC^{G/P}$ on $G/P$. \end{thm}

\section{Cone structures of parabolic geometries}\label{s.parabolic}

\subsection{Cartan geometries}\label{ss.Cartan}
We recall some background on Cartan connections; we refer the reader to \cite{csbook, Sharpe} for detailed introductions.

\begin{defin}\label{Def_Cartan_geom} Let $G$ be a complex Lie group and $P\leq G$ a complex Lie subgroup and denote by $\mfg$ and $\mfp$ their respective Lie algebras. A (holomorphic) \emph{Cartan geometry} of type $(G,P)$ on a complex manifold $M$
is given by:
\begin{itemize}
\item a (holomorphic) principal $P$-bundle $\mcG\rightarrow M$, and
\item a \emph{Cartan connection}, that is a $\mfg$-valued $1$-form $\omega\in\Omega^1(\mcG, \mfg)$ with the following properties:
\begin{itemize}
\item[(a)] $\omega$ induces a trivialisation $T\mcG\cong \mcG\times \mfg$ of the tangent bundle of $\mcG$;
\item[(b)] $\omega$ reproduces the generators of the fundamental vector fields of the $P$-action on $\mcG$: $\omega(\zeta_X)=X$ for all $X\in\mfp$; and
\item[(c)] $\omega$ is $P$-equivariant: $(r^p)^*\omega=\textrm{Ad}(p^{-1})\circ \omega$ for all $p\in P$,
\end{itemize}
\end{itemize}
where $r^p:\mcG\rightarrow\mcG$ denotes the principal right-action by $p\in P$ on $\mcG$ and $\zeta_X$ the fundamental vector field on $\mcG$ generated by $X\in\mfp$.
\end{defin}

The \emph{homogeneous model} of a Cartan geometry of type $(G,P)$ is the homogeneous space $G/P$ with its natural Cartan geometry given by the projection $G\rightarrow G/P$ and the (left) Maurer--Cartan form
$\omega_G\in\Omega^1(G, \mfg)$ on $G$.

\begin{defin} A \emph{morphism (resp. isomorphism) of Cartan geometries} $(\mcG\rightarrow M, \omega)$ and $(\mcG'\rightarrow M', \omega)$ of type $(G,P)$ is a principal bundle morphism $\Phi:\mcG\rightarrow\mcG'$ (resp. isomorphism)
such that $\Phi^*\omega'=\omega$ (the latter implies that $\Phi$ and its base map $M\rightarrow M'$ are local biholomorphisms).
\end{defin}

A Cartan connection $\omega$ induces an isomorphism of vector bundles
\begin{equation}\label{tangent_bundle}
TM\cong \mcG\times_P\mfg/\mfp,
\end{equation}
where the action of $P$ on $\mfg/\mfp$ is induced from the adjoint representation of $G$. In view of \eqref{tangent_bundle} a Cartan geometry of type $(G,P)$ on a manifold $M$ may be viewed
as a geometric structure that makes $M$ look infinitesimally like $G/P$. The integrability conditions for a Cartan geometry to be locally isomorphic (not just infinitesimally) to $G/P$ is given by the vanishing of the curvature of the Cartan connection.

\begin{defin}\label{Cartan_curvature}
Suppose $(\mcG\rightarrow M, \omega)$ is a Cartan geometry. Then its curvature $K\in\Omega^2(\mcG, \mfg)$ is given by $K(\xi,\eta)=d\omega(\xi,\eta)+[\omega(\xi), \omega(\eta)]$ for vector fields $\xi, \eta$ on $\mcG$. The defining properties of the Cartan connection imply that its curvature $K$ is horizontal and $P$-equivariant. Thus we can view the curvature as a section $\kappa$ of $$\Wedge^2 T^*M\otimes (\mcG\times_P\mfg) \cong \mcG\times_P (\Wedge^2(\mfg/\mfp)^*\otimes\mfg).$$ By an abuse of notation, we will also write $\kappa: \mcG\rightarrow \Wedge^2(\mfg/\mfp)^*\otimes\mfg$ for the corresponding $P$-equivariant function such that
$$\kappa(X+\mfp, Y+\mfp)=K(\omega^{-1}(X), \omega^{-1}(Y))=[X,Y]-\omega([\omega^{-1}(X), \omega^{-1}(Y)]).$$
The
\emph{torsion} of the Cartan geometry is the element of $\Omega^2( \mcG, \mfg/\mfp)$ given by the composition of the projection $\mfg\rightarrow \mfg/\mfp$ with $K$.
The geometry $(\mcG\rightarrow M, \omega)$ is said to be \emph{torsion-free} if $K$ has values in $\mfp$, and \emph{flat} if $K$ vanishes identically.
\end{defin}

For a proof of the following result, see \cite[Proposition 1.5.2]{csbook} or \cite[Theorem 5.1]{Sharpe}:

\begin{prop}\label{p.flatdevelop}
A Cartan geometry of type $(G,P)$ is flat if and only if it is locally isomorphic to its homogeneous model $(G\rightarrow G/P, \omega_G)$.
\end{prop}

\subsection{Parabolic geometries}\label{ss.parageom}
We will be particularly interested in Cartan geometries infinitesimally modelled on rational homogeneous spaces, called parabolic geometries. We recall here briefly some basics on the theory of parabolic geometries; for details we refer to
\cite{csbook}.

\begin{defin}\label{d.parageom}
A \emph{(holomorphic) parabolic geometry} on a complex manifold $M$ is a (holomorphic) Cartan geometry $(\mcG\rightarrow M, \omega)$ on $M$ of type $(G,P)$, where
$G$ is a complex semisimple Lie group and $P\leq G$ is a parabolic subgroup.
The tangent bundle
of $M$ is filtered by vector subbundles according to \eqref{filtration} as follows:
\begin{equation}\label{omega_induced_filt}
TM=T^{-k}M\supset ... \supset T^{-1}M\quad\textrm{ with }\quad T^{-i}M\cong\mcG\times_P\mfp^{-i}/\mfp \quad\textrm{ for } i<0,
\end{equation}
and one obtains induced filtration on all tensor bundles.
\end{defin}

\begin{defin}\label{adjoint_tractor}
In Definition \ref{d.parageom}, the vector bundle $\mcA M:=\mcG\times_P\mfg$ is called the \emph{adjoint bundle} of the parabolic geometry $(\mcG\rightarrow M,\omega)$. We write
\begin{equation}\label{filtration_adjoint_tractors}
\mcA M=\mcA^{-k}M\supset...\supset \mcA^0M\supset...\supset \mcA^{k}M.
\end{equation}
for the filtration by vector subbundles induced by \eqref{filtration}. Note that, since $\mcA M/\mcA^0M\cong TM$ and $T^*M\cong \mcA^1 M$ by Lemma \ref{l.basic_isos}, one has a natural projection $\mcA M\rightarrow TM$ and a natural  inclusion $T^*M\hookrightarrow \mcA M$. We write $\textrm{gr}(\mcA M)=\textrm{gr}(T^*M)\oplus \textrm{gr}_0(\mcA M) \oplus \textrm{gr}(TM)$ for the associated graded vector of the filtered bundle $\mcA M$.
\end{defin}

\begin{notcon}
Given a parabolic geometry, we write $\mcP_0:=\mcG/P_+\rightarrow M$ for the $P_0$-principal bundle given by the natural projection, where $P_0$ is the Levi subgroup of $P$ and $P_+$ its unipotent radical. Since $P_+$ acts trivially on the subsequent quotients of \eqref{filtration},
the Cartan connection induces isomorphisms
\begin{equation} \label{gr_A}
\textrm{gr}_i(\mcA M)\cong \mcP_0\times_{P_0} \mfp_i \quad\quad \textrm{ for all } -k\leq i \leq k.
\end{equation}
In particular, we also have
\begin{equation}\label{gr_T}
\textrm{gr}(TM)\cong \mcP_0\times_{P_0} \mfp_-\quad \textrm{ and }\quad  \textrm{gr}(T^*M)\cong \mcP_0\times_{P_0} \mfp_+.
\end{equation}
Since the Lie bracket on $\mfg$ is $P_0$-invariant, it  induces a vector bundle homomorphism
\begin{equation}\label{curl_bracket}
\{\cdot, \cdot\}:\textrm{gr}(\mcA M)\times\textrm{gr}(\mcA M)\rightarrow\textrm{gr}(\mcA M)
\end{equation}
making $\textrm{gr}(\mcA M)$ into a bundle of graded Lie algebras. We write $\{\cdot, \cdot\}$ also for its restriction to $\textrm{gr}(TM)\times \textrm{gr}(TM)\rightarrow \textrm{gr}(TM)$.\end{notcon}

\begin{defin}
A parabolic geometry is called \emph{regular}, if the filtration \eqref{omega_induced_filt} makes $M$ into a filtered manifold and the Levi-bracket $\mcL$ of \eqref{omega_induced_filt} coincides with $\{\cdot, \cdot\}$ (that is, the symbol algebra coincides with $\mfp_-$.)
\end{defin}

\begin{notcon}
Recall that the curvature $\kappa$ of a parabolic geometry (as a special case of Cartan geometry) is  a section of the vector bundle
$$ \Wedge^2 T^* M \otimes \mcA M \cong \mcG\times_P(\Wedge^2(\mfg/\mfp)^*\otimes\mfg) \cong \mcG\times_P (\Wedge^2\mfp_+\otimes\mfg),$$
which is equipped with a filtration (by homogeneity of maps between filtered spaces) induced by \eqref{omega_induced_filt} and \eqref{filtration_adjoint_tractors}. If $\kappa$ is of homogeneity $\geq\ell$ (that is, $\kappa$ is a section of the $\ell$-th filtration component of $\Wedge^2 T^* M \otimes \mcA M$), then we write $\textrm{gr}_\ell(\kappa)$ for the projection of $\kappa$ to the $\ell$-th grading component $\textrm{gr}_{\ell}( \Wedge^2 T^* M \otimes \mcA M)$ of the associated graded
$\textrm{gr}( \Wedge^2 T^* M \otimes \mcA M)$ of $\Wedge^2 T^* M \otimes \mcA M$. \end{notcon}

The next proposition follows easily from the definition of the curvature:

\begin{prop}\emph{\cite[Corollary 3.1.8]{csbook}}\label{Prop_regularity} Suppose $(\mcG\rightarrow M, \omega)$ is a parabolic geometry. Then $M$ equipped with the filtration \eqref{omega_induced_filt} is a filtered manifold if and only if
$\kappa$ is of homogeneity $\geq 0$. If this is the case, then $\emph{gr}_0(\kappa)=\{\cdot, \cdot\}-\mcL$. Hence, the geometry is regular if and only if $\kappa$ is of homogeneity $\geq 1$.
\end{prop}

\begin{notcon}
The boundary operators \eqref{d.homology} induce
bundle maps
\begin{align*}
\partial^*: \Wedge^\ell T^*M\otimes\mcA M&\rightarrow  \Wedge^{\ell-1} T^*M\otimes\mcA M\\
\partial^*: \Wedge^\ell \textrm{gr}(T^*M)\otimes\textrm{gr}(\mcA M)&\rightarrow  \Wedge^{\ell-1} \textrm{gr}(T^*M)\otimes\textrm{gr}(\mcA M).
\end{align*}
\end{notcon}

\begin{defin}\label{normalization}
A parabolic geometry is called \emph{normal}, if $\partial^*\kappa=0$.
The \emph{harmonic curvature}  $\hat\kappa$ of a normal parabolic geometry is the image of the projection of $\kappa$ to $\textrm{Ker}(\partial^*)/\textrm{Im}(\partial^*)=\mcG\times_PH_2(\mfp_+,\mfg)$.
By Proposition \ref{p.Kostant}, we can identify the  vector bundle
$\mcG\times_P H_*(\mfp_+, \mfg)\cong\mcP_0\times_{P_0} H_*(\mfp_+, \mfg)$  with a subbundle of $\Wedge^* \textrm{gr}(T^*M)\otimes\textrm{gr}(\mcA M)$ and consider the harmonic curvature as a section of the latter vector bundle on $M$.
\end{defin}

Usually when parabolic geometries arise from geometric problems, we can make choices such that they are normal in the sense of Definition \ref{normalization}. 
If the parabolic geometry is regular and normal the harmonic curvature is still a complete obstruction to flatness due to Bianchi-identities for $\kappa$:

\begin{prop}\emph{\cite[Theorem 3.1.12]{csbook}} \label{harmcurv}
Suppose that $(\mcG\rightarrow M,\omega)$ is a regular normal
parabolic geometry. Then $\kappa$ vanishes on an open set $U\subset M$ if and only if $\hat\kappa$ vanishes on $U$.
\end{prop}

\subsection{Cone structures and conic connections associated to parabolic geometries}\label{ss.cones_induce_para}
 From now on, we assume that $\mfp$ is a maximal parabolic subalgebra of a simple Lie algebra $\mfg$ 
 determined by a long simple root $\alpha$ and let $P\subset G$ be as in Section \ref{s.homogenous}.
 
 \begin{defin}\label{d.associated}
Given a regular parabolic geometry $(\pi_M: \mcG \to M, \omega)$ of type $(G,P)$, we have a natural structure of a filtered manifold on $M$ with the constant symbol $\mfp_-$ by \eqref{omega_induced_filt} 
and an $\mcC^{G/P}_o$--isotrivial cone structure subordinate to it: $$ \mcC := \mcG \times_P \mcC^{G/P}_o\cong \mcP_0\times_{P_0} \mcC^{G/P}_o\subset \mathbb P(T^{-1}M)\subset \mathbb P(TM).$$
This is called the \emph{  cone structure associated to the parabolic geometry}.
\end{defin}

Conversely, we have:

\begin{thm}\label{Cartan_cones_str}
Suppose $(\mathfrak{g},\alpha)\neq (A_n,\alpha_1/\alpha_n)$. Let $(M, \{T^{-i}M\})$  
be a filtered manifold 
of dimension $\dim(G/P)$ with symbol algebra
$\mathfrak{p}_-$ equipped with a subordinate 
$\mathcal{C}_o^{G/P}$-isotrivial cone structure $\mathcal{C}$ such that 
\begin{equation}\label{comp.cond.}
\mcL (u, v) =0 \quad \emph{ for any } u, v \in T^{-1}_x M,\, x \in M,\, \emph{ with } u \in \widehat{\mcC}_x \emph{ and } v \in \widehat{T}_u \mcC_x,
\end{equation} 
where $\mcL: T^{-1} M \otimes T^{-1} M \to T^{-2} M/T^{-1} M$ is the Levi bracket.
Then there exists a unique
regular normal parabolic geometry $(\mathcal{G} \rightarrow
M,\omega)$ of type $(G,P)$ inducing $\{T^{-i}M\}$ via
\eqref{omega_induced_filt} such that $\mathcal{C}=\mathcal{G} \times_P
\mathcal{C}_o^{G/P}$. Moreover, this association gives rise to an equivalence of categories between such cone structures $(M, \{T^{-i}M\}, \mcC)$ and regular normal parabolic geometries of type $(G,P)$.
\end{thm}
\begin{proof}
The Levi subgroup $P_0$ of $P$ equals the group $\textrm{Aut}_{\textrm{gr}}(\mfg)$ of grading preserving Lie algebra automorphisms of $\mfg$. Hence, via restriction, we have group inclusions
\begin{equation*}
P_0=\textrm{Aut}_\textrm{gr}(\mfg)\stackrel{i}{\hookrightarrow} \textrm{Aut}_\textrm{gr}(\mfp_-)\stackrel{j}{\hookrightarrow} \textrm{GL}(\mfp_{-1}),
\end{equation*}
where $j$ is an inclusion, since $\mfp_{-1}$ generates the Lie algebra $\mfp_-$. Suppose now $(M, \{T^{-i}M\})$ is a filtered manifold with symbol algebra $\mfp_-$ and consider 
its the frame bundle $\textrm{Fr}(\textrm{gr}(TM))$, which, by dint of $j$, can be also viewed a subbundle of the frame bundle $\textrm{Fr}(T^{-1}M)$ of $T^{-1}M$.

We show first that reductions of structure group of the frame bundle $\textrm{Fr}(\textrm{gr}(TM))$ corresponding to $i$ are equivalent to subordinate $\mcC_o^{G/P}$-isotrivial cone structures satisfying \eqref{comp.cond.}.
Given a reduction of $\textrm{Fr}(\textrm{gr}(TM))$ to $P_0\leq \textrm{Aut}_\textrm{gr}(\mfp_-)$, that is, a principal $P_0$-subbundle $\mcP_0\leq \textrm{Fr}(\textrm{gr}(TM))$, then 
 $\mcC=\mcP_0\times_{P_0}\mcC_{o}^{G/P}$ defines a subordinate $\mcC_o^{G/P}$-isotrivial cone structure. Moreover, for $0\neq w\in\mfg_{\alpha}=\mfq_{-1}^F$ one has, by Lemma \ref{nested_parabolics} and Corollary 
 \ref{FF_model},
 $$T_w\hat\mcC_{o}^{G/P}=\mfq_{-1}^F\oplus [\mfp_0, \mfq_{-1}^F]=\mfq_{-1}^F\oplus [\mfq_{-1}^V, \mfq_{-1}^F]=\mfq_{-1}^F\oplus\mfq_{-2}.$$
By (3) of Proposition \ref{long_versus_short_root}, we have $[\mfq_{-1}^F, \mfq_{-2}]=0$ and hence $\mcC$ in particular satisfies \eqref{comp.cond.}.
Conversely, suppose $\mcC$ is a subordinate $\mcC_o^{G/P}$-isotrivial cone structures satisfying \eqref{comp.cond.}. As a subgroup of $\textrm{GL}(\mfp_{-1})$, one has $P_0\cong \textrm{Aut}(\hat\mcC_o^{G/P})$
and hence $\mcC$ gives rise to a reduction of structure group $\mcP_0\leq \textrm{Fr}(T^{-1}M)$ corresponding to $P_0\leq  \textrm{GL}(\mfp_{-1})$. By \cite[Proposition 6 and 7]{Hwang-Mok02}, the property \eqref{comp.cond.} implies that for any $x\in M$ the group of grading preserving Lie algebra automorphism of $\textrm{gr}(T_xM)$ that in addition preserve $\hat \mcC_x$ is isomorphic to $P_0\leq  \textrm{Aut}_\textrm{gr}(\mfp_-)$.
Hence $\mcP_0$ can be naturally viewed as a $P_0$-principal subbundle of $\textrm{Fr}(\textrm{gr}(TM))\leq \textrm{Fr}(T^{-1}M)$. The claim of the theorem now follows from
Theorem 3.1.14 and Observation 3.1.7 of \cite{csbook}.
\end{proof}

Recall that, by Proposition \ref{harmcurv}, the cone structure $(M, \{T^{-i}M\}, \mcC)$ associated to a regular normal parabolic geometry  is locally
isomorphic to $(G/P, \{T^{-i}G/P\}, \mcC^{G/P})$ if and only if the harmonic curvature (viewed as a $P$-equivariant function) $\hat \kappa: \mcG\rightarrow H_{2}(\mfp_+, \mfg)$ of the corresponding normal parabolic geometry vanishes. 
Regularity of $\omega$ (a vacuous  condition in the symmetric case) implies that $\hat\kappa$ has actually values in the subspace
$\bigoplus_{r\geq 1}H_{2,r}(\mfp_+, \mfg)$. 
Thus we have the following as an immediate consequence of Proposition \ref{p.vanishing}.

\begin{cor}\label{flat_cases}
Assume that $(\mfg, \alpha)$ does not belong to the following three classes:
\begin{itemize}
\item  symmetric type \item  contact type  \item $(B_n/D_{n+1}, \alpha_3)$ for $n\geq 4$.  \end{itemize} Then a regular normal parabolic geometry $(\mcG \to M, \omega)$ of type $(G,P)$ is  flat. In particular, the associated $\mcC_o^{G/P}$--isotrivial cone structure is locally isomorphic to $\mcC^{G/P}$ on $G/P$.
\end{cor}

In the three cases excluded in Corollary \ref{flat_cases}, we would like to characterize flat parabolic geometries in terms of the invariants of the associated cone structures. To start with, the following is straightforward.

\begin{prop}\label{p.correspondence}
In Definition \ref{d.associated}, we can write
$$\mcC=\mcG\times_P \mcC_o^{G/P}\cong \mcG\times_P P/Q\cong \mcG/Q,$$
where $Q$ is defined as in Lemma \ref{lem_highest_weight_cone},
and  the natural projection $\pi_\mcC: \mcG\rightarrow \mcG/Q$ defines a $Q$-principal bundle over $\mcC\cong \mcG/Q$.
\begin{itemize} \item[(1)]
We can view $\omega$ as a Cartan connection of type
$(G,Q)$ on $\mathcal{G}$. Thus
$(\pi_\mathcal{C}: \mathcal{G} \rightarrow \mathcal{C},
\omega)$ is a parabolic geometry of type $(G,Q)$
on $\mathcal{C}$.
\item[(2)] The Cartan connection $\omega$ induces isomorphisms $$TM\cong\mcG\times_P\mfg/\mfp \mbox{ and } T\mcC\cong\mcG\times_Q\mfg/\mfq,$$
and  the filtrations $\{T^{-i}M\}$ on $TM$  and $T^{-i}\mcC\cong \mcG\times_Q\mfq^{-i}/\mfq$ on $T\mcC$. \item[(3)]
The tangent map $Tp: T\mcC\rightarrow TM$ of $p:\mcC\rightarrow M$ is given by the natural projection
$$Tp:\mcG\times _Q\mfg/\mfq\rightarrow \mcG\times_ P\mfg/\mfp,$$ which implies $\mcV\cong\mcG\times_Q\mfp/\mfq$
for the vertical bundle of $p$. \end{itemize} \end{prop}

\begin{rem}
In the terminology of \cite{CapCores}, we say that $(\pi_\mathcal{C}: \mathcal{G} \rightarrow \mathcal{C}, \omega)$ in Proposition \ref{p.correspondence} is the \emph{correspondence space} of the
parabolic geometry $(\pi_M:\mathcal{G} \rightarrow M, \omega)$ with respect to $Q\leq P$.
\end{rem}

\begin{prop}\label{omega_induced_filtration}
In Proposition \ref{p.correspondence}, consider the filtration $\{\mathcal{D}^{-i}\}$ of Definition
\ref{d.cone_filtration} on $T\mathcal{C}$. Then $\omega$ yields
\begin{itemize}
\item[(1)]   an identification of
$\mcD^{-i}$ with $T^{-i}\mcC$ for each $1\leq i\leq 4$
such that $(Tp)^{-1}(T^{-1}M)=T^{-4}\mcC$ (in the
symmetric case, $T\mcC=T^{-4}\mcC=T^{-3}\mcC$);
 \item[(2)]
 an identification of
$\{\cdot, \cdot\}:\mathcal{V} \times
\emph{gr}_{-i}(T\mathcal{C})\rightarrow
\emph{gr}_{-i-1}(T\mathcal{C})$ with the corresponding
Levi brackets $\mathcal{L}$; and
\item[(3)]  a conic connection
    $\mathcal{F}\cong\mathcal{G}\times_Q
    (\mathfrak{q}_{-1}^F\oplus
    \mathfrak{q})/\mathfrak{q}\subset T^{-1}\mathcal{C}$
on $\mathcal{C}$. \end{itemize}
\end{prop}

\begin{proof}
Since $Q$ is the stabiliser in $P$ of $\mfq_{-1}^F$, it follows (similarly as in Proposition \ref{long_versus_short_root}) that $\mcD^{-1}=T^{-1}\mcC$. Since $\mcV\cong\mcG\times_Q\mfp/\mfq$, this also
shows that $\mcF$ defines a conic connection on $\mcC$, thus (3) holds. As in  Proposition \ref{long_versus_short_root}, we also have $(Tp)^{-1}(T^{-1}M)=T^{-4}\mcC$. For the remaining claims in (1) and (2) note that the preimages under $T\pi_\mcC : T \mcG \to T \mcC$
of the distributions $T^{-i}\mcC$ are the distributions $\omega^{-1}(\mfq^{-i})\subset T\mcG$ on $\mcG$ and locally any section of $T^{-i}\mcC$ lifts to a section of $\omega^{-1}(\mfq^{-i})$.
Recall from Definition \ref{Cartan_curvature} that the curvature of a Cartan connection is horizontal. Since $\omega$ is a Cartan connection of type $(G,P)$ over $M$, its curvature $K\in\Omega^2(\mcG, \mfg)$ must therefore
vanish upon insertion of sections of $\omega^{-1}(\mfp)$ (corresponding to $\mcV\subset T\mcC$). Hence, for sections $\xi$ of $\omega^{-1}(\mfp)$ and $\eta$ of $\omega^{-1}(\mfq^{-i})$ we have
$$0=K(\xi, \eta)=\xi\cdot \omega(\eta)-\eta\cdot\omega(\xi)-\omega([\xi, \eta])+[\omega(\xi), \omega(\eta)].$$
Since $[\mfp, \mfq^{-i}]=\mfq^{-i-1}$ for $i=1,2,3$ by Proposition \ref{long_versus_short_root}, this implies that $T^{-i-1}\mcC=[\mcV, T^{-i}\mcC]+T^{-i}\mcC$ for $1\leq i\leq 3$ and claim (2).
Since $T^{-1}\mcC=\mcD^{-1},$ we have $\mcD^{-i}=T^{-i}\mcC$ for $2\leq i\leq 4$, which establishes (1).
\end{proof}

We note that \cite[Prop. 1.2.1]{LeBrunThesis} (see Ex. \ref{ex.path.geom}), and \cite[Thm. 3.1.16  and Prop. 4.1.5]{csbook} imply:

\begin{prop}\label{p.path.geom} Assume $(\mfg,\alpha)=(A_n, \alpha_1 \textrm{ or } \alpha_n)$. Associating to a regular normal parabolic geometry of type $(G,P)$ the conic connection $\mcF$ of Proposition \ref{omega_induced_filtration}
on its associated cone structure $\mcC=\mathbb P(TM)$ gives rise to an equivalence of categories between regular normal parabolic geometries of type $(G,P)$ on a complex manifold $M$ and conic connections on $\mathbb P(TM)$.
\end{prop}

\begin{notcon}\label{n.kappa}
In Proposition \ref{p.correspondence},
the curvature of $\omega$ as a function $\kappa: \mcG\rightarrow \Wedge^2 \mfg^*\otimes\mfg$ is the same for the parabolic geometries of type $(G,P)$
on $M$ and type $(G,Q)$ on $\mcC$. It gives rise to a function
$\mathcal{G}\rightarrow
    \Wedge^2
    (\mathfrak{g}/\mathfrak{p})^*\otimes\mathfrak{g}.$
    This function viewed as a $P$-equivariant function can
    be identified with a section of
    $\Wedge^2T^*M\otimes\mathcal{A} M$, which we denote by
    $\kappa^M.$  On the other hand, when viewed as a
    $Q$-equivariant function, it can be identified with a
    section of
    $\Wedge^2(T\mathcal{C}/\mathcal{V})^*\otimes\mathcal{A}
    \mathcal{C}$, which we denote by
    $\kappa^\mathcal{C}$.
    \end{notcon}

As a consequence of Lemma \ref{l.rel.hom} and Proposition \ref{p.rel.Lie} we have:

\begin{prop}\label{norm_cone}
In the Notation \ref{n.kappa} we have:
\begin{enumerate}
\item[(1)] $\kappa^\mcC=\iota\circ \kappa^M$ as functions on $\mcG$, where $\iota$ is as in Lemma \ref{l.rel.hom}, and 
the normality of $(\pi_M: \mcG\rightarrow M, \omega)$ implies the normality of $(\pi_\mcC: \mcG\rightarrow \mcC, \omega)$.
\item[(2)] If we write $\hat\kappa^M: \mcG\rightarrow H_{2}(\mfp_+,\mfg)$ and $\hat\kappa^{\mcC}: \mcG\rightarrow H_{2}(\mfg_+,\mfg)$ for the harmonic curvatures %functions) 
of the normal parabolic geometries $(\pi_M: \mathcal{G} \rightarrow M, \omega)$ and $(\pi_\mathcal{C}: \mathcal{G} \rightarrow \mathcal{C}, \omega)$ respectively, then they are related as
\begin{equation}
\hat\kappa^\mcC=\emph{proj}\circ \hat\kappa^M: \mcG\rightarrow H_0(\mfq_{1}^V, H_{2}(\mfp_+,\mfg))\subset H_2(\mfq_+,\mfg),
\end{equation}
where $\emph{proj}:  H_{2}(\mfp_+,\mfg)\rightarrow  H_{2}(\mfp_+,\mfg)/\mfq_{1}^VH_{2}(\mfp_+,\mfg)= H_0(\mfq_{1}^V, H_{2}(\mfp_+,\mfg))$ is the natural projection.
\end{enumerate}

\end{prop}

\begin{rem} In contrast to $(\mcG\rightarrow M, \omega)$ the induced parabolic geometry $(\mcG\rightarrow\mcC, \omega)$ on $\mcC$
is in general not regular: by Proposition \ref{harm_curv_C}, the subspace $H_0(\mfq_{1}^V, H_{2, 1}(\mfp_+, \mfg))\subset \Wedge^2\mfq_+\otimes\mfg$
is contained in grading components of non-positive degree (namely, $0$ or $-1$ (with respect to the grading induced by $\mfq$ on $\Wedge^2\mfq_+\otimes\mfg$)). Hence, $\hat{\kappa}^\mcC$ and thus also $\kappa^\mcC$ have in general
values in non-positive grading components of $\Wedge^2\mfq_+\otimes\mfg$, which implies that  $(\mcG\rightarrow\mcC, \omega)$ is (in general) not regular by Proposition \ref{Prop_regularity}.
We shall see below that its regularity is equivalent to the existence
of a characteristic connection on $\mcC$.
\end{rem}

\begin{notcon}\label{n.kappaharm}
Recall  that the natural projection $\mcQ_0:=\mcG/Q_+\rightarrow \mcC$ is a $Q_0$-principal bundle
over $\mcC$ and that the Cartan connection
$\omega$ induces an isomorphism
$$\Wedge^2\textrm{gr}(T\mcC)^*\otimes \textrm{gr}(\mcA\mcC)\cong \mcQ_0\times_{Q_0}\Wedge^2\mfq_+\otimes\mfg.$$
Hence the $Q_0$-module $H_0(\mfq_{1}^V, H_{2}(\mfp_+, \mfg))\subset \Wedge^2\mfp_+\otimes \mfg\subset\Wedge^2\mfp_+\otimes \mfg$ defines a subbundle of
$\Wedge^2\textrm{gr}(T\mcC/\mcV)^*\otimes \textrm{gr}(\mcA\mcC)$. In the sequel, we view $\hat\kappa^\mcC$ alternately as a $Q_0$-equivariant function
$\mcQ_0\rightarrow H_0(\mfq_{1}^V, H_{2}(\mfp_+, \mfg))$ as well as a section of $\Wedge^2\textrm{gr}(T\mcC/\mcV)^*\otimes \textrm{gr}(\mcA\mcC)$.
\end{notcon}

\begin{thm}\label{m1}
Suppose $(\mfg,\alpha)$ is a simple Lie algebra with a choice of long simple root. Let $(\mcG\rightarrow M, \omega)$ be a regular normal parabolic geometry of type $(G, P)$ and let $(M, \{T^{-i}M\}, \mcC)$ be the associated filtered manifold with symbol algebra $\mfp_-$
equipped with a subordinate $\mcC_o^{G/P}$--isotrivial cone structure.
Let $(\mcG\rightarrow\mcC, \omega)$ be the normal parabolic geometry on $\mcC$ and let $\mcF$ be the
conic connection on $\mcC$ of Proposition \ref{omega_induced_filtration}.
Then the following statements are equivalent.
\begin{itemize}
\item[(1)] $(\mcG\rightarrow\mcC, \omega)$ is regular.
\item[(2)] $\mcF$ is characteristic (i.e. $\tau_\mcF=0$).
\item[(3)] $\mcC$ admits a characteristic conic connection (i.e. $\tau^\mcC=0$).
\end{itemize}
Assume furthermore that  $(\mfg,\alpha)$ is different from $(B_n/D_{n+1},\alpha_1)$ for $n\geq 2$ or $(A_n, \alpha_2/\alpha_{n-1})$ for $n\geq 3$ or $(A_n, \alpha_1/\alpha_{n})$ for $n\geq 2$.
Then under the equivalent conditions (1)--(3), the parabolic geometry $(\mcG \to M, \omega)$ is flat and $(M, \{T^{-i}M\}, \mcC)$ is locally isomorphic to $(G/P, \{T^{-i} G/P\}, \mcC^{G/P})$.
\end{thm}
\begin{proof} Note first that conditions (1)-(3) are trivially satisfied for $(A_n, \alpha_1/\alpha_{n})$ for $n\geq 2$.
By Corollary \ref{flat_cases}, 
we may therefore assume that $(\mfg, \alpha)$ belongs to one of the following two cases, which we handle separately.

{\bf Case I: When $(\mfg, \alpha)\neq (A_n, \alpha_1/\alpha_{n})$ is of symmetric type.}
By Proposition \ref{p.depth}, the depth of $\mfq$ equals $3$. Recall that $\kappa^\mcC$ is a section of $\Wedge^2 (T\mcC/\mcV)^*\otimes \mcA\mcC$. 
The lowest filtration component of $\Wedge^2 (T\mcC/\mcV)^*\otimes \mcA\mcC$ is of degree $0$
and the grading component of degree $0$ of the associated graded is given by
$$\textrm{gr}_0(\Wedge^2 (T\mcC/\mcV)^*\otimes \mcA\mcC)=\mcF^*\otimes \textrm{gr}_{-2}(T\mcC)^*\otimes \textrm{gr}_{-3}(T\mcC).$$
Hence, $\kappa^\mcC$ is of homogeneity $\geq 0$ and, by Proposition \ref{Prop_regularity}, we have
$$\textrm{gr}_0(\kappa^\mcC)=\{\cdot, \cdot\}-\mcL =-\mcL\vert_{\mcF\times \textrm{gr}_{-2}(T\mcC)}=\tau_\mcF,$$
where the second equality follows from (3) of Proposition \ref{long_versus_short_root}. By Proposition \ref{Prop_regularity},
it  follows that (1) is equivalent to (2). Evidently, (2) implies (3). To see that (3) implies (2), note that by (i) of Proposition \ref{harm_curv_C} and Proposition \ref{omega_induced_filtration} we
have the following decomposition
\begin{equation}\label{dec}
\mcF^*\otimes \textrm{gr}_{-2}(T\mcC)^*\otimes \textrm{gr}_{-3}(T\mcC)=\textrm{Ker}(\partial^*)\oplus\textrm{Im}(\partial)=\mcH\oplus \mcF^*\otimes {\rm II}(\mcV),
\end{equation}
where $\mcH=\mcQ_0\times_{Q_0} H_{0}(\mfq_1^V, H_{2,1}(\mfp_+, \mfg))$ and ${\rm II}$ as in Definition \ref{d.tau_C}.
By normality, $\textrm{gr}_0(\kappa^\mcC)=\tau_\mcF$ is a section of $\textrm{Ker}(\partial^*)$. Hence, \eqref{dec} shows
that under the natural identification of the quotient
\eqref{e.quotient} with $\textrm{Ker}(\partial^*)$ one also has $\textrm{gr}_0(\kappa^\mcC)=\tau^\mcC$.
Therefore, by Proposition \ref{p.torsion}, (3) implies (2), which completes the proof that (1)-(3) are equivalent.

Now let us prove the last statement of the theorem in (Case I).
The decomposition \eqref{dec} shows also that
$\textrm{gr}_0(\kappa^\mcC)$ equals the homogeneous component of degree $0$ of the harmonic curvature $\kappa^\mcC$, which, by (2) of Proposition \ref{p.vanishing}, equals the entire harmonic curvature
$\hat\kappa^\mcC$ except for $(\mfg,\alpha)=(B_n/D_{n+1},\alpha_1)$ for $n\geq 2$ or $(A_n, \alpha_2/\alpha_{n-1})$ for $n\geq 3$ or the excluded case $(A_n, \alpha_1/\alpha_{n})$.
Hence, except for these cases, the parabolic geometry is flat under the equivalent assumptions (1)-(3).

{\bf Case II: When $(\mfg,\alpha)$ is of contact type or equal to $(B_n/D_{n+1}, \alpha_3)$ for $n\geq 4$.}
By Proposition \ref{p.depth}, the depth of $\mfq$ is $5$ in the contact case (resp. $6$ in the case of $(B_n/D_{n+1}, \alpha_3)$) with $\mfp^{\pm 1}=\mfq^{\pm 4}$.

Let us consider $\kappa^M$ and $\kappa^\mcC=\iota\circ\kappa^M$ as functions on $\mcG$ with valued in $\Wedge^2\mfp_+\otimes\mfg\subset \Wedge^2\mfq_+\otimes\mfg$. Recall that, when we say $\kappa^M$ or $\kappa^\mcC$ is of homogeneity $\geq \ell$ (i.e has values in the $\ell$-th filtration component of $\Wedge^2\mfp_+\otimes\mfg$), we mean with respect to the filtration on $\mfg$ induced by $\mfp$ in the first case and by $\mfq$ in the second case.
Since $\omega$ is regular as a Cartan connection of type $(G,P)$, Proposition \ref{Prop_regularity} shows that $\kappa^M$ is of homogeneity $\geq 1$
and, by Theorem 3.1.12 of \cite{csbook}, $\textrm{gr}_1(\kappa^M)$ coincides with the homogeneous component $\hat\kappa_1^M$ of degree $1$ of $\hat\kappa^M$.
By Proposition \ref{p.vanishing}, we have $\hat\kappa^M=\hat\kappa_1^M$ and the function has values in the subspace $H_{2, 1}(\mfp_+, \mfg)$ of
$\Wedge^2\mfp_1\otimes\mfp^{-1}/\mfp\cong \Wedge^2\mfp_1\otimes\mfp_{-1}$. It follows that $\kappa^\mcC$ is of homogeneity $\geq -1$ and that
\begin{equation*}
\textrm{gr}_{-1}(\kappa^\mcC)=\textrm{proj}\circ\textrm{gr}_{1}(\kappa^M)=\textrm{proj}\circ \hat\kappa^M=\hat\kappa^\mcC,
\end{equation*}
where $\textrm{proj}: \Wedge^2\mfp_1\otimes \mfp^{-1}/\mfp\rightarrow (\Wedge^2\mfp_1\otimes \mfp^{-1}/\mfp)/\mfq_{1}^V(\Wedge^2\mfp_1\otimes \mfp^{-1}/\mfp)\cong \mfq_{1}^F\otimes\mfq_2\otimes\mfq_{-4}$ is the natural projection.
Let us now view $\hat\kappa^\mcC=\textrm{gr}_{-1}(\kappa^\mcC)$ as a section of $\mcF^*\otimes\textrm{gr}_{-2}^*(T\mcC)\otimes \textrm{gr}_{-4}(T\mcC)$. By Proposition \ref{Prop_regularity}, (see also the proof of Proposition \ref{Prop_regularity} in \cite[Cor. 3.1.8]{csbook}), the fact that $\kappa^\mcC$ is of homogeneity $\geq -1$ implies that $[\mcF, T^{-2}\mcC]\subset T^{-4}\mcC$ and $\textrm{gr}_{-1}(\kappa^\mcC)$ equals the negative of the map
$$\mcF\times \textrm{gr}_{-2}(T\mcC)\rightarrow \textrm{gr}_{-4}(T\mcC)$$ induced by the Lie bracket of vector fields. Hence, the parabolic geometry is flat
if and only if $\hat\kappa^\mcC=\textrm{gr}_{-1}(\kappa^\mcC)=0$, which is the case if and only if $[\mcF, T^{-2}\mcC]\subset T^{-3}\mcC$. By Proposition \ref{long_versus_short_root}, the previous equivalent
conditions are also equivalent to $\mcF$ being characteristic  (that is, $[\mcF, T^{-2}\mcC]\subset T^{-2}\mcC$) and by Proposition \ref{Prop_regularity} also to regularity of  $(\mcG\rightarrow\mcC, \omega)$.

Since (2) evidently implies $(3)$,
to finish the proof it remains to show that  $(3)$ implies $\hat\kappa^\mcC=0$.
To see this, assume that $\mcC$ admits a characteristic conic connection, say $\mcF'$. Then $T^{-1}\mcC=\mcF'\oplus\mcV$ and,
by Proposition \ref{omega_induced_filtration}, this implies $[T^{-1}\mcC, T^{-2}\mcC]\subset T^{-3}\mcC$. In particular, also $[\mcF, T^{-2}\mcC]\subset T^{-3}\mcC$, which implies $\hat\kappa^\mcC=0$.
\end{proof}

By Proposition \ref{p.torsion} and Lemma \ref{FF.injective}, Theorem \ref{m1} implies:

\begin{cor}\label{c.F.unique} 
Assume the setting of Theorem \ref{m1} and that $(\mfg,\alpha)\neq (A_n, \alpha_1/\alpha_{n})$. If $\mcC$ admits a characteristic conic connection $\mcF'$, then the conic connection $\mcF$ on $\mcC$ 
of Proposition \ref{omega_induced_filtration} is characteristic and $\mcF'=\mcF$.
\end{cor}

In the three cases excluded in Theorem \ref{m1} with the exception of $(A_2, \alpha_1/\alpha_2)$ and $(B_2,\alpha_1)$, 
we shall prove that, under the assumption that $\mcF$ is characteristic, the cubic torsion $\chi_\mcF$ equals $\hat \kappa^\mcC$. 
To do so we need to compute the harmonic curvature $\hat \kappa^\mcC$. An effective way to compute the harmonic curvature of a normal parabolic geometry is in terms of a Weyl structure. 
Hence, we will make now a short digression on Weyl structures.

\subsection{Digression on Weyl structures}\label{ss.Weyl}
We record some facts about the theory of Weyl structures for parabolic geometries and refer for more details to Sections 5.1 and 5.2 in \cite{csbook}. Throughout this section we assume that
$N$ is a complex manifold equipped with a regular normal parabolic geometry $(\mcG\rightarrow N, \omega)$ of type $(G,Q)$, where the groups are as in Lemma \ref{lem_highest_weight_cone}.
Setting $\mcQ_0:=\mcG/Q_+$, the natural projections $\mcG\rightarrow \mcQ_0$
and $\mcQ_0\rightarrow N$ define a $Q_+$-principal bundle over $\mcQ_0$ and a $Q_0$-principal bundle over $N$, respectively.

\begin{defin}\label{Weyl_struc_def}(cf. \cite[Def. 5.1.1]{csbook}.) Any (local) $Q_0$-equivariant section $\sigma$ of $\mcG\rightarrow \mcQ_0$ is called a \emph{(local) Weyl structure} for $(\mcG\rightarrow N, \omega)$.
\end{defin}

\begin{rem}
In contrast to parabolic geometries in the smooth category, in the holomorphic category only local Weyl structures in general exist. For readability we nevertheless write a Weyl structure, by abuse of notation, as a map $\sigma: \mcQ_0\rightarrow \mcG$. All objects associated to it should be understood as being only locally defined.
\end{rem}

The freedom in the choice of a Weyl structure is as follows:

\begin{prop} \emph{\cite[Proposition 5.1.1]{csbook}}\label{freedom_Weyl}
Suppose  $\sigma:\mcQ_0\rightarrow\mcG$  is a local Weyl structure. Then any other local Weyl structure $\hat\sigma:\mcQ_0\rightarrow\mcG$ is of the form
$$\hat\sigma=\sigma \cdot \exp(\Upsilon_1)...\exp(\Upsilon_k)$$ for unique local $Q_0$-equivariant functions $\Upsilon_i:\mcQ_0\rightarrow\mfq_i$ for $i=1,...k$ (defining local sections of $\emph{gr}_i(T^*N)\cong \mcQ_0\times_{Q_0}\mfq_{i}$), where $\exp: \mfq_+\rightarrow Q_+$ is the usual Lie algebra exponential. Here, $k$ is the depth of $\mfq$.
\end{prop}

Moreover, a choice of Weyl structure gives rise to the following data:
\begin{prop}\emph{\cite[Proposition and Definition 5.1.2]{csbook}}\label{Interpr_weyl}
Let $\sigma: \mcQ_0\rightarrow \mcG$ be a local Weyl structure of $(\mcG\rightarrow N, \omega)$.
Then $\sigma^*\omega\in\Omega^1(\mcQ_0, \mfg)$ is $Q_0$-equivariant and hence decomposes according to $\mfg=\mfq_{-k}\oplus...\oplus\mfq_0\oplus...\oplus \mfq_k$ as
\begin{equation}
\sigma^*\omega=\sigma^*\omega_{-k}+...+\sigma^*\omega_{0}+...+\sigma^*\omega_{k}.
\end{equation}
The components have the following interpretation:
\begin{enumerate}
\item[(1)] For $i<0$ the form $\sigma^*\omega_{i}$ descends to a bundle map $TN\rightarrow \emph{gr}_{i}(TN)\cong\mcQ_0\times_{Q_0}\mfq_{i}$ whose restriction to $T^iN$ equals the natural projection
$q_i: T^iN\rightarrow \emph{gr}_{i}(TN)$. Hence, $\Theta_{i}:=\sigma^*\omega_{i}=q_i\circ \pi_i$ for a unique projection $\pi_i: TN\rightarrow T^iN$ that restricts to the identity on $T^iN$ and
\begin{align}
\Theta&:=\Theta_{-1}+...+\Theta_{-k}: TN\rightarrow \emph{gr}_{-1}(TN)\oplus...\oplus \emph{gr}_{-k}(TN)\\\nonumber
&\Theta(\xi)=\pi_{-1}(\xi)+q_{-2}(\pi_{-2}(\xi))+...+q_{-k}(\xi)
\end{align}
defines an isomorphism $TN\cong \emph{gr}(TN)$, called the soldering form associated to $\sigma$.
\item[(2)] The form $\gamma:=\sigma^*\omega_0\in \Omega^1(\mcQ_0, \mfq_0)$ defines a principal connection on $\mcQ_0$, called the Weyl connection associated to $\sigma$.
As such it induces linear connections $\nabla$ on all vector bundles associated to $\mcQ_0$. In particular, it induces a linear connection on $\emph{gr}_{-i}(TN)\cong \mcQ_0\times _{Q_0}\mfq_{-i}$ and
hence on $TN$ via $\Theta$.
\item[(3)] The form $\sigma^*\omega_{1}+...+\sigma^*\omega_{k}$ descends to an element $\emph{\textsf{P}}\in\Omega^1(TN, \emph{gr}(T^*N))$, called the Schouten tensor
associated to $\sigma$. Via $\Theta$, we may view $\emph{\textsf{P}}$ also as a section of  $\emph{gr}(T^*N)\otimes \emph{gr}(T^*N)$.
\end{enumerate}
\end{prop}

Since $Q_0$ acts by grading preserving Lie algebra automorphism on $\mfg$, Proposition \ref{Interpr_weyl} implies:

\begin{cor}\label{cor_levi_parallel}
Let $\sigma$ be a Weyl structure for $(\mcG\rightarrow N, \omega)$ and consider $\emph{gr}(\mcA N)\cong \mcQ_0\times_{Q_0}\mfg$. Then the section $\{\cdot, \cdot\}$ of $\Wedge^2\emph{gr}(\mcA N)^*\otimes \emph{gr}(\mcA N)$  defined as in \eqref{curl_bracket}
is parallel with respect to $\nabla$. In particular, by Proposition \ref{Prop_regularity}, we also have $\nabla\mcL=0$ for the Levi-bracket $\mcL$ of the filtered manifold $(N, \{T^{-i}N, \})$.
\end{cor}

\begin{notcon}\label{n.sigma}
Having fixed a local Weyl structure $\sigma$, we can consider the pullback $\sigma^*K\in\Omega^2(\mcQ_0, \mfg)$ of the curvature $K$ of $(\mcG\rightarrow N, \omega)$.
Since $\sigma^*K(_-,_-)=d\sigma^*\omega(_-,_-)+[\sigma^*\omega(_-), \sigma^*\omega(_-)]$ is $Q_0$-equivariant and horizontal, $\sigma^*K$ descends to a $2$-form $\kappa^\sigma$ on $N$ with values in
\begin{equation}\label{Hilfs_eq}
\textrm{gr}(\mcA N)=\textrm{gr}(TN)\oplus\textrm{gr}_0(\mcA N)\oplus\textrm{gr}(T^*N),
\end{equation}
which, via $\Theta: TN\cong \textrm{gr}(TN)$, can be also viewed as a section of
\begin{equation}\label{bundle_Hodge}
\Wedge^2 \textrm{gr}(TN)^*\otimes\textrm{gr}(\mcA N)\cong \mcQ_0\times_{Q_0}\Wedge^2\mfq_+\otimes\mfg=\textrm{Im}(\partial)\oplus\textrm{Ker}(\square)\oplus \textrm{Im}(\partial^*),
\end{equation}
where the three bundles on the right-hand side are induced from the corresponding $Q_0$-modules in the Hodge decompsotion \eqref{Hodge} of $\Wedge^2\mfq_+\otimes\mfg$.
By assumption, $\kappa^\sigma$ is a section of the subbundle $\textrm{Ker}(\partial^*)=\textrm{Ker}(\square)\oplus \textrm{Im}(\partial^*)\subset \Wedge^2 \textrm{gr}(TN)^*\otimes\textrm{gr}(\mcA N)$.
\end{notcon}

Recall that the action of $\mfq_+$ on $\Wedge^2\mfq_+\otimes\mfg$ maps $\ker(\partial^*)\subset\Wedge^2\mfq_+\otimes\mfg$ to $\textrm{im}(\partial^*)\subset \Wedge^2\mfq_+\otimes\mfg$, implying that $\mfq_+$ acts trivially on
$H^2(\mfq_+, \mfg)=\ker(\partial^*)/\textrm{im}(\partial^*)$. Therefore, Proposition \ref{freedom_Weyl} implies that
the part of $\kappa^\sigma$ in $\textrm{Ker}(\square)$ is independent of the choice of $\sigma$ and evidently must equal the harmonic curvature $\hat\kappa$ of $(\mcG\rightarrow N, \omega)$. Thus we conclude:

\begin{prop}\label{p.reduce}
If the $\textrm{Ker}(\square)$-part of $\kappa^\sigma$ is zero for some Weyl structure $\sigma$, then $\hat\kappa$ is zero. \end{prop}

Note that $\Theta+\gamma\in\Omega^2(\mcQ_0,\mfq_-\oplus\mfq_0)$ (for a fixed Weyl structure $\sigma$) defines a Cartan connection of type $(\exp(\mfq_-)\rtimes Q_0, Q_0)$ on $\mcQ_0\rightarrow N$.
The individual components of $\kappa^\sigma$ according to the decomposition \eqref{Hilfs_eq} can be described as certain curvature quantities associated to the Cartan connection $\Theta+\gamma$
and the harmonic curvature as the $\textrm{Ker}(\square)$-part of these quantities,
see Theorem 5.2.9 of \cite{csbook} for a precise formulation. In particular, \cite[Theorem 5.2.9]{csbook} shows:

\begin{prop}\label{key_prop} Let $\sigma$ be a local Weyl structure of  $(\mcG\rightarrow N, \omega)$. Then
the component of $\kappa^\sigma$ in $\Wedge^2\emph{gr}(TN)^*\otimes \emph{gr}(TN)\cong \Wedge^2 T^*N\otimes \emph{gr}(TN) $ is given by
\begin{equation}\label{key_equation}
\emph{\textsf{T}}+(\partial \emph{\textsf{P}}\cap \Wedge^2 \emph{gr}(TN)^*\otimes \emph{gr}(TN))\in \emph{Ker}(\partial^*)\cap \Wedge^2\emph{gr}(TN)^*\otimes \emph{gr}(TN),
\end{equation}
where \begin{itemize}
\item[(1)] $\emph{\textsf{T}} \in \Wedge^2\emph{gr}(TN)^*\otimes \emph{gr}(TN)$ is the torsion of the Cartan connection $\Theta+\gamma$ as defined in Definition \ref{Cartan_curvature}; \item[(2)]
$\partial: \emph{gr}(TN)^*\otimes \emph{gr}(\mcA N)\rightarrow \Wedge^2 \emph{gr}(TN)^*\otimes \emph{gr}(\mcA N)$ is  the bundle map induced from $\partial$ as in \eqref{Hodge}; and \item[(3)] $ \partial \emph{\textsf{P}}\cap \Wedge^2 \emph{gr}(TN)^*\otimes \emph{gr}(TN)$ denotes the component of $\partial \emph{\textsf{P}}$ belonging to $\Wedge^2 \emph{gr}(TN)^*\otimes \emph{gr}(TN)$. \end{itemize}
 In particular, in view of \eqref{bundle_Hodge},
the component of $\hat\kappa$ inside $\Wedge^2\emph{gr}(TN)^*\otimes \emph{gr}(TN)$ equals the $\emph{Ker}(\square)$-component of $\emph{\textsf{T}}$.
\end{prop}

The formula for $\textsf{T} \in\Wedge^2 T^*N\otimes \textrm{gr}(TN)\cong \Wedge^2 \textrm{gr}(TN)^*\otimes \textrm{gr}(TN)$ can be easily computed:

\begin{lem}\emph{\cite[Lemma 5.1.2 ]{csbook}}\label{torsion_formula}
Let $\sigma$ be a local Weyl structure of $(\mcG\rightarrow N, \omega)$ and $\Theta$ the corresponding soldering form.
Using (\ref{curl_bracket}),  one has
\begin{equation}
\textsf{T}(\xi, \eta)=\nabla_\xi \Theta(\eta)-\nabla_\eta\Theta(\xi)-\Theta([\xi, \eta])+\{\Theta(\xi), \Theta(\eta)\}
\end{equation}
for local vector fields  $\xi$ and $\eta$ on $N$. In particular, the component of $T$ inside $\Omega^2(N, \emph{gr}_{-i}(TN))$ is given by
\begin{equation}
\nabla_\xi \Theta_{-i}(\eta)-\nabla_\eta\Theta_{-i}(\xi)-\Theta_{-i}([\xi, \eta])+\sum_{a,b<0, a+b=-i}\{\Theta_{a}(\xi), \Theta_{b}(\eta)\}.
\end{equation}
\end{lem}

It will be convenient to use a particular class of Weyl structures:

\begin{prop}\label{special_Weyl}
Suppose $(\mcG\rightarrow N, \omega)$ is a regular normal parabolic geometry of type $(G,Q)$ and consider the line bundle $\mcF=\mcG\times_Q (\mfq_{-1}^F\oplus\mfq)/\mfq\cong \mcQ_0\times_{Q_0}\mfq_{-1}^F$.
For any non-vanishing local section $f$ of $\mcF$ there exists a (non-unique) Weyl structure such that the induced linear connection $\nabla$ on $\mcF$ has the property that $\nabla_\xi f=0$ for  any section $\xi$ of $T^{-2}N$.
\end{prop}
\begin{proof}
Let $\alpha\in\Delta^0$ be the long simple root such that $\mfq_{-1}^F=\mfg_{-\alpha}$ and denote by $B$ the Killing form of $\mfg$. Let $H_{\alpha}$ be the unique element in the Cartan subalgebra $\mfh$
such that  $-\alpha=-B(H_{\alpha},_- )$, which we view as a functional on $\mfq_0$ (defining the $1$-dimensional representation of $\mfq_0$ on $\mfq_{-1}^F$ ). From the description of $\mfq$ in terms of simple roots in Lemma \ref{lem_highest_weight_cone}
and the fact that the Lie bracket induces an isomorphism $\mfq_{1}^F\otimes\mfq_1^V\cong\mfq_2$ we conclude that
$\frac{2\langle \alpha, \beta\rangle}{\langle \alpha, \alpha\rangle}=\frac{2\beta(H_{\alpha})}{\langle \alpha, \alpha\rangle}$
equals $-1$  (resp. $1$)  for any root $\beta$ corresponding to a roots space in $\mfq_{1}^V$ (resp. $\mfq_2$)  as in  the proof of Proposition \ref{long_versus_short_root}. Hence,
the restriction of $\textrm{ad}(H_\alpha)$ to $\mfq_{1}\oplus\mfq_2$ is injective. Therefore, for $i=1, 2$ and a fixed nonzero element $F\in\mfq_{-1}^F$ ,the map
\begin{align}\label{special_con}
\mfq_i\times\mfq_{-i}&\rightarrow\mfq_{-1}^F\\\nonumber
(Z, X)&\mapsto-\alpha([Z,X])F=-B(H_{\alpha}, [Z, X])F=B(X, [Z, H_\alpha])F
\end{align}
is non-degenerate, since so is $B: \mfq_i\times\mfq_{-i}\rightarrow\C$  by Lemma \ref{l.basic_isos}. Hence, for $i=1,2$ and any local non-vanishing section $f$ of $\mcF$, the corresponding bundle map $\textrm{gr}_{i}(T^*N)\times\textrm{gr}_{-i}(TN)\rightarrow \mcF$, given by
$(\Upsilon_i, \xi_{-i})\mapsto -\{\{ \Upsilon_i,, \xi_{-i}\}, f\}$, must be non-degenerate as well. Suppose now $\sigma: \mcQ_0\rightarrow\mcG$ is a local Weyl structure and $f$ a local non-vanishing section of $\mcF$.
Then there exists a unique local section $\Upsilon_1$ of $\textrm{gr}_{1}(T^*N)$ such that $\nabla_{\xi_{-1}}f= \{\{ \Upsilon_1,\xi_{-1}\}, f\}$ for all local sections $\xi_{-1}$ of $T^{-1}N=\textrm{gr}_{-1}(TN)$. By Proposition 5.1.6 of \cite{csbook},
for the Weyl structure $\hat\sigma=\sigma\exp(\Upsilon_1)$, we then have $$\widehat{\nabla}_{\xi_{-1}}f=\nabla_{\xi_{-1}}f-\{\{ \Upsilon_1, \xi_{-1}\}, f\}=0,$$
for all local sections $\xi_{-1}$ of $T^{-1}N$. Via the soldering form $\widehat \Theta$ corresponding to $\hat{\sigma}$,  we may identify $T^{-2}N\cong \textrm{gr}_{-1}(TN)\oplus \textrm{gr}_{-2}(TN)$ and write any
local section $\xi$ of $T^{-2}N$ as $\xi=\xi_{-1}+\xi_{-2}$ accordingly. Then there exists a unique local section $\Upsilon_2$ of $\textrm{gr}_{2}(T^*N)$
such that $$\widehat{\nabla}_\xi f=\widehat{\nabla}_{\xi_{-2}}f=\{\{ \Upsilon_2, \xi_{-2}\}, f\}$$ for all local sections $\xi$ of $T^{-2}N$. By Proposition 5.1.6 of \cite{csbook}, for the Weyl structure $\tilde{\sigma}=\hat\sigma\exp(\Upsilon_2)$,
we then have $\widetilde{\nabla}_\xi f=\widehat{\nabla}_\xi f-\{\{ \Upsilon_2, \xi_{-2}\}, f\}=0$ for all local sections $\xi$ of $T^{-2}N$. Hence, $\tilde{\sigma}$ is a Weyl structure with the desired property.
\end{proof}

\begin{rem}
The line bundle $\mcF$ of the previous proposition is not a bundle of scales in the sense of Definition 5.1.4 of \cite{csbook} (except for $(\mfg,\alpha)=(A_n,\alpha_1/\alpha_n)$), since $\textrm{ad}(H_\alpha)$ is only injective on $\mfq_{1}\oplus\mfq_2$ and not on $\mfq_+$.Hence,  having fixed a local non-vanishing section $f$ of $\mcF$, Corollary 5.1.6 of \cite{csbook} can not be applied to ensure the existence of a (unique) Weyl structure such that $\nabla_\xi f=0$ for all vector fields $\xi$.
\end{rem}

We end this digression by fixing some more notation.

\begin{notcon} Suppose $(\mcG\rightarrow N, \omega)$ is a parabolic geometry of type $(G,Q)$ as above and let $\sigma:\mcQ_0\rightarrow \mcG$ be a local Weyl structure.
Then, via the isomorphism $\Theta: TN\cong \textrm{gr}(TN)$, we can identify $\kappa^\sigma$ with a section of  the graded vector bundle
$$\Wedge^2\textrm{gr}(TN)^*\otimes\textrm{gr}(\mcA N)=\Wedge^2\textrm{gr}(T^*N)\otimes\textrm{gr}(\mcA N).$$ We write $\kappa^\sigma_i$ for the $i$-th grading component of $\kappa^\sigma$. Similarly, we can identify $\textsf{T}$ and $\textsf{P}$ with sections of the graded vector bundles $\Wedge^2\textrm{gr}(T^*N)\otimes \textrm{gr}(TN)$ and $\textrm{gr}(T^*N)\otimes \textrm{gr}(T^*N)$ respectively. We write $\textsf{T}_i$ and $\textsf{P}_i$ for their $i$-th grading component. Note that
$\textsf{P}_i=0$ for $i\leq 1$, since all grading components of  $\textrm{gr}(T^*N)$ are postive.
\end{notcon}

\subsection{Cubic torsion of the conic connection associated to a parabolic geometry}\label{ss.proofchi}

In the setting of Theorem \ref{m1}, if the conic connection $\mcF$ has vanishing characteristic torsion, we can expresses the cubic torsion as the harmonic curvature $\hat\kappa^\mcC$ of $(\mcG\rightarrow \mcC, \omega)$.

\begin{thm}\label{t.chi} In the setting of Theorem \ref{m1} assume that $(\mfg, \alpha)$ does not equal $(B_2, \alpha_1)$ nor $(A_2,\alpha_1/\alpha_2)$. Furthermore, 
assume that $\mcF$ is characteristic (i.e. $\tau_\mcF=0$). Then $\hat\kappa^\mcC=\chi_\mcF$.
\end{thm}

\begin{proof}
 By Theorem \ref{m1}, we may assume that $(\mfg, \alpha)$ belongs to the cases (3) and (5) of Proposition \ref{p.vanishing}.
We write $\kappa=\kappa^\mcC$ for the curvature and $\hat\kappa=\hat\kappa^\mcC$ for the harmonic curvature of $(\mcG \to \mcC, \omega)$. Since the vanishing of $\tau_{\mcF}$ implies that  $(\mcG \to \mcC, \omega)$ is regular  by Theorem \ref{m1}, we see that $\hat\kappa$ is of positive homogeneity. By Proposition \ref{harm_curv_C}, and (3) and (5) of Proposition \ref{p.vanishing}, this implies that
$\hat\kappa$ is a section of the bundle $$\textrm{Ker}(\square)\cap\mcF^*\otimes \textrm{gr}_{-2}(T\mcC)^*\otimes \mcV\subset \textrm{gr}_{-1}(T\mcC)^*\otimes \textrm{gr}_{-2}(T\mcC)^*\otimes T^{-1}\mcC. $$
In particular, $\hat\kappa$ is of homogeneous degree $2$ (as a section of the grading vector bundle
$\Wedge^2\textrm{gr}(T\mcC)^*\otimes\textrm{gr}(\mcA\mcC)$). Hence, Theorem 3.1.12 of \cite{csbook} implies that $\kappa\in\Wedge^2T^*\mcC\otimes\mcA\mcC$ is of homogeneity $\geq 2$ and that
$\textrm{gr}_{2}(\kappa)=\hat\kappa$.

For any local Weyl structure $\sigma$ of $(\mcG\rightarrow \mcC, \omega)$ one has $\textrm{gr}_{2}(\kappa^\mcC)=\kappa^\sigma_2$ (by Proposition \ref{freedom_Weyl} and the fact that the action of $\mfq_+$ on $\Wedge^2\mfq_+\otimes \mfg$ is of homogeneity $\geq 1$). Therefore, we have by Proposition \ref{key_prop}
\begin{equation}\label{formula_harm_curv}
\hat\kappa=\textrm{gr}_{2}(\kappa^\mcC)=\kappa^\sigma_2=\textsf{T}_2+\partial \textsf{P}_2\in \textrm{Ker}(\square)\cap \mcF^*\otimes \textrm{gr}_{-2}(T\mcC)^*\otimes \mcV.
\end{equation}
We will compute $\hat\kappa=\kappa^\sigma_2$ in terms of a local Weyl structure $\sigma$ as in Proposition \ref{special_Weyl}. To do so
fix a local non-vanishing section $f$ of $\mcF\cong \mcQ_0\times_{Q_0}\mfq_{-1}^F$ and assume that $\sigma$ is a local Weyl structure such that $\nabla_ \xi f=0$ for the corresponding Weyl connection $\nabla$ and all sections $\xi$ of $T^{-2}\mcC$. Then the formula in Lemma \ref{torsion_formula} restricted to the component of $\textsf{T}_2$ inside $\mcF^*\otimes \textrm{gr}_{-2}(T\mcC)^*\otimes T^{-1}\mcC$  shows
\begin{equation}\label{formula_T_2}
(T_2+\partial \textsf{P}_2)(f, q_{-2}(\xi))=\nabla_f\pi_{-1}(\xi)-\pi_{-1}([f,\xi])+\partial \textsf{P}_2(f, q_{-2}(\xi))
\end{equation}
for any local section $\xi$ of $T^{-2}\mcC$. Since the curvature is of homogeneity $\geq 2$, we have moreover $\kappa^\sigma_1=0$. Hence, Proposition \ref{key_prop} shows that $\textsf{T}_1=0$, since
$\textsf{P}_i=0$ for $i\leq 1$ and $\partial$ is grading-preserving. In particular, the components of $\textsf{T}_1$ inside $\mcF^*\otimes \mcV^*\otimes T^{-1}\mcC$ and
$\mcF^*\otimes \textrm{gr}_{-2}(T\mcC)^*\otimes \textrm{gr}_{-2}(T\mcC)$ vanish. By Lemma \ref{torsion_formula} this gives the following identities
\begin{align}
&0=T_1(f,v)=\nabla_fv-\pi_{-1}([f,v]),\label{T11}\\
&0=T_1(f, q_{-2}(\xi))=\nabla_f q_{-2}(\xi)-q_{-2}\pi_{-2}([f,\xi])+\{f,\pi_{-1}(\xi)\}\label{T12},
\end{align}
for any section $v$ of $\mcV$ and $\xi$ of $T^{-2}\mcC$. Recall that $\textrm{gr}_{-2}(T\mcC)\cong \mcF\otimes \mcV$ via the Levi bracket $\mcL$. For any section $\xi$ of $T^{-2}\mcC$ let us write $v_\xi$
for the unique section of $\mcV$ such that $q_{-2}(\xi)=\mcL(f,v_\xi)$, then \eqref{T11} implies that
\begin{equation}\label{pi_1}
\pi_{-1}(\xi)=\xi-[f,v_\xi]+\pi_{-1}([f,v_\xi])=\xi-[f,v_\xi]+\nabla_fv_\xi.
\end{equation}
If we insert $\xi=[f,v]$ for some $v\in\mcV$ into \eqref{T12}, then \eqref{T11} and Corollary \ref{cor_levi_parallel} imply
\begin{align*}
0&=\nabla_f\mcL(f,v)-q_{-2}\pi_{-2}([f,[f,v])+\{f,\pi_{-1}([f,v])\}\\
&=\mcL(f, \nabla_fv)-q_{-2}\pi_{-2}([f,[f,v])+\mcL(f, \nabla_fv).
\end{align*}
Since $\mcF$ is characteristic, $[f,[f,v]]$ is a section of $T^{-2}\mcC$ and we can drop the projection $\pi_{-2}$ in the second term, which implies
\begin{equation}\label{nabla_fv}
\mcL(f,\nabla_fv)=\frac{1}{2} q_{-2}([f,[f,v]]).
\end{equation}
To compare the expression \eqref{formula_T_2} with $\chi_\mcF$, we need to compute  \eqref{formula_T_2} under the isomorphism
$$\mcF^*\otimes \textrm{gr}_{-2}(T\mcC)^*\otimes \mcV\cong S^3\mcF^*\otimes \textrm{End}(\mcV, \textrm{gr}_{-2}(T\mcC))$$
induced by \eqref{isos}.
To do that we insert $\xi=[f,v]$ for some $v\in\mcV$ into \eqref{formula_T_2} and take the Levi-bracket with $f$.
By \eqref{T11} and \eqref{nabla_fv} this leads for the first term in \eqref{formula_T_2} that:

\begin{equation}\label{first_term}
\mcL(f, \nabla_f\pi_{-1}([f,v]))= \mcL(f, \nabla_f\nabla_fv)=\frac{1}{2}q_{-2}([f,[f,\nabla_fv]]).
\end{equation}
For the second term in \eqref{formula_T_2} note that for $\eta=[f,[f,v]]$ we have $v_\eta=2\nabla_fv$ by \eqref{nabla_fv} and hence \eqref{pi_1} and the identity \eqref{first_term} imply

\begin{align}
&\mcL(f, \pi_{-1}([f,[f,v]]))=q_{-2}([f,\pi_{-1}([f,[f,v]])])\nonumber\\
&=q_{-2}([f,[f,[f,v]]])-2q_{-2}([f,[f,\nabla_fv]])+2q_{-2}([f, \nabla_f\nabla_f v])\nonumber\\
&=q_{-2}([f,[f,[f,v]]])-q_{-2}([f,[f,\nabla_fv]])\label{term_2}.
\end{align}

Subtracting \eqref{term_2} form \eqref{first_term} gives $\tilde{\chi}_\mcF$ as defined in \eqref{tildechi},
since $w(f,v)$ as in \eqref{w(f,v)} equals $\nabla_fv$. In view of \eqref{formula_harm_curv} and \eqref{bundle_Hodge}, $\hat\kappa$ hence equals the part of $\tilde{\chi}_\mcF$ in $\textrm{Ker}(\partial^*)$ (which automatically lies in
$\textrm{Ker}(\square)$ by \eqref{formula_harm_curv}).
By (ii) of Proposition \ref{harm_curv_C} we have $$\textrm{Ker}(\partial^*)\cap \mcF^*\otimes \textrm{gr}_{-2}(T\mcC)^*\otimes \mcV\cong S^3\mcF^*\otimes \textrm{Hom}_0(\mcV, \textrm{gr}_{-2}(\mcD)),$$ which shows that $\hat\kappa=\chi_\mcF$.
\end{proof}

Combining Theorems \ref{m1} and \ref{t.chi}, we obtain the following.

\begin{cor}\label{coroll} Assume $(\mfg,\alpha)$ is a simple Lie algebra with a long simple root, excluding $(B_2, \alpha_1)$ and $(A_2,\alpha_1/\alpha_2)$, and denote by $\mfp$ the corresponding parabolic subalgebra.
Let $(\mcG \to M, \omega)$ be a regular normal parabolic geometry of type $(G,P)$. Assume that the natural conic connection $\mcF$ on the associated cone structure $\mcC$ has vanishing characteristic torsion, or equivalently, the parabolic geometry $(\mcG \to \mcC, \omega)$ is regular. Then $(\mcG \to M, \omega)$ is flat if and only if $\chi_{\mcF} =0.$ \end{cor}

As an application, we have the following result, which is a stronger version of Theorem \ref{t.Mok}.

\begin{thm}\label{t.sMok}
Let $(\mfg, \mfp)$ be as in Corollary \ref{coroll} and let $G/P$ be the corresponding rational homogeneous space.
Let $M$ be a complex manifold of dimension equal to $\dim G/P$ with
 a bracket-generating distribution $\mcH \subset T M$ of rank $\dim(\mfp_{-1})$. Suppose we have a $\mcC_o^{G/P}$-isotrivial  cone structure  $\mcC \subset \mbP \mcH$ subordinate to $\mcH$ with a conic connection $\mcF \subset T \mcC$ satisfying $\tau_{\mcF} = \chi_{\mcF} =0$. Then the cone structure is locally isomorphic to the cone structure $\mcC^{G/P}$ of $G/P$. \end{thm}

 \begin{proof}
If $(\mfg,\alpha)=(A_n, \alpha_1/\alpha_n)$ for $n\geq 3$, then the theorem follows directly from Proposition \ref{p.path.geom} and Corollary \ref{coroll}.
Now assume $(\mfg,\alpha)\neq (A_n, \alpha_1/\alpha_n)$. Then the assumption $\tau_{\mcF} =0$ and the bracket-generating assumption on $\mcH$ imply that the tangential filtration $\{T^{-i}M\}$ on $M$ given by the weak derived system of 
$\mcH$ has the constant symbol of type $\mfp_-$, by Proposition 5 of \cite{Hwang-Mok02} and Proposition \ref{p.Hw12}. Moreover, Proposition \ref{p.Hw12} and Theorem \ref{Cartan_cones_str} imply that $(\{T^{-i}M\}, \mcC)$ gives rise to 
a regular normal parabolic geometry of type $(G,P)$ on $M$ inducing $(\{T^{-i}M\}, \mcC)$. By Corollary \ref{c.F.unique}, $\mcF$ equals the conic connection on $\mcC$ of Proposition \ref{omega_induced_filtration} and hence
Corollary \ref{coroll} completes the proof.  \end{proof}

\newpage
\section{Appendix}\label{Appendix}
In the following tables $\textrm{Gr}(p,q)$ denotes the Grassmannian variety of $p$-planes in $\C^{p+q}$ and $\textrm{LGr}(n,n)$ the Langrange--Grassmanian variety of Langrangean subspaces of the
symplectic vector space $\C^{2n}$.
\begin{table}[h!]
\centering
\small
\begin{tabular}{ |p{1cm}|p{2cm}|p{2.5cm}|p{2cm}|p{2cm}|p{4cm}| }
% \hline
 %\multicolumn{4}{|c|}{Cone Structures of Irreducible Parabolic Geometries} \\
 \hline
 $\mfg$ & $\mfp$ & $G/P$&$\mfq$ &$\mfp_0^{ss}$ & $\mcC_o^{G/P}\hookrightarrow \mathbb P(\mfp_{-1})$\\
 \hline
 &&&&&\\
 $A_n$, $n\geq 1$&\dynkin[parabolic=1]{A}{}& $\mathbb{P}^{n}$&\dynkin[parabolic=3]{A}{}& $A_{n-1}$ & $\mathbb{P}^{n-1}$\\
 &&&&&\\
 $A_n$, $\tiny{n\geq 3}$  &   \begin{dynkinDiagram}{A}{*.*x*.*}
 \dynkinBrace[p]{1}{3}
 \dynkinBrace[q-1]{4}{5}
 \end{dynkinDiagram}
 &
$ \textrm{Gr}(p,q)$
& \begin{dynkinDiagram}{A}{*.xxx.*}
 \dynkinBrace[p-1]{1}{2}
 \dynkinBrace[q-1]{4}{5}
 \end{dynkinDiagram}   &$A_{p-1}\times A_{q-1}$  & $\mathbb{P}^{p-1}\times  \mathbb{P}^{q-1}\hookrightarrow \mathbb{P}^{pq-1} $ Segre \\
 &&&&&\\
 $B_n$, $n\geq 2$& \dynkin[parabolic=1]{B}{}  & $Q^{2n-1}$ quadric &\dynkin[parabolic=3]{B}{} & $B_{n-1}$ & $Q^{2n-3}\hookrightarrow \mathbb{P}^{2n-2}$ by $\mcO(1)$\\
 &&&&&\\
 $C_n$, $n\geq 3$ & \dynkin[parabolic=16]{C}{} & $\textrm{LGr}(n,n)$ &  \begin{dynkinDiagram}{C}{**.*xx}
 \end{dynkinDiagram}
&$A_{n-1}$& $\mathbb{P}^{n-1}\hookrightarrow \mathbb{P}(S^2\C^n)$ Veronese\\
 &&&&&\\
 $D_n$, $n\geq 3$  &\dynkin[parabolic=1]{D}{} & $Q^{2n-2}$ quadric &\dynkin[parabolic=3]{D}{}& $D_{n-1}$ &  $Q^{2n-4}\hookrightarrow \mathbb{P}^{2n-3}$ by $\mcO(1)$  \\
 &&&&&\\
 $D_n$, $n\geq 5$ &\dynkin[parabolic=32]{D}{} & $\mathbb S^n$ spinor variety  & \begin{dynkinDiagram}{D}{**.*x*x}
 \end{dynkinDiagram}
&$A_{n-1}$& $\textrm{Gr}(2,n-2)\hookrightarrow \mathbb{P}(\Wedge^2\C^n)$ Pl\"ucker \\
&\dynkin[parabolic=16]{D}{}&& \begin{dynkinDiagram}{D}{**.*xx*}
 \end{dynkinDiagram}
&&\\
&&&&&\\
 $E_6$&   \dynkin[parabolic=32]{E}{6}   & $\mathbb {OP}^2$ (complexified) Cayley plane &\begin{dynkinDiagram}{E}{****xx}
 \end{dynkinDiagram}
&$D_5$& $\mathbb{S}^5\hookrightarrow \mathbb{P}^{15}$ by $\mcO(1)$ \\
 &&&&&\\
 $E_7$&  \dynkin[parabolic=64]{E}{7}  & exceptional  &\begin{dynkinDiagram}{E}{*****xx}
 \end{dynkinDiagram}
 &$E_6$& $\mathbb {OP}^2\hookrightarrow\mathbb{P}^{26}$ Severi\\
 \hline
\end{tabular}
\caption{Irreducible compact Hermitian symmetric spaces and their VMRT-structures}
\label{tab1}
\end{table}

\begin{table}[h!]
\centering
\small
\begin{tabular}{ |p{1cm}|p{2.5cm}|p{2.5cm}|p{2cm}|p{5cm}| }
% \hline
 %\multicolumn{4}{|c|}{Cone Structures of Irreducible Parabolic Geometries} \\
 \hline
 $\mfg$ & $\mfp$ &$\mfq$ &$\mfp_0^{ss}$ & $\mcC_0^{G/P}\hookrightarrow \mathbb P(\mfp_{-1})$\\
 \hline
 &&&&\\
 $B_n$, $n\geq 4$& \begin{dynkinDiagram}{B}{**x*.**}\end{dynkinDiagram}  &
 \begin{dynkinDiagram}{B}{*xxx*.**}\end{dynkinDiagram} & $A_2\times B_{n-3}$ & $\mathbb P^2\times Q^{2n-7}\hookrightarrow \mathbb{P}^{6n-16}$ Segre \\
&&&&\\
 $D_n$, $n\geq 5$  &\begin{dynkinDiagram}{D}{**x*.***}\end{dynkinDiagram} &\begin{dynkinDiagram}{D}{*xxx*.***}\end{dynkinDiagram}& $A_2\times D_{n-3}$ &  $\mathbb P^2\times Q^{2n-8}\hookrightarrow \mathbb{P}^{6n-22}$ Segre  \\
 \hline
\end{tabular}
\caption{$(B_n/D_n, \alpha_3)$ and their VMRT-structures}
\label{tab2}
\end{table}

\begin{table}[h!]
\centering
\small
\begin{tabular}{ |p{1cm}|p{2.5cm}|p{2.5cm}|p{2cm}|p{5cm}| }
% \hline
 %\multicolumn{4}{|c|}{Cone Structures of Irreducible Parabolic Geometries} \\
 \hline
 $\mfg$ & $\mfp$ &$\mfq$ &$\mfp_0^{ss}$ & $\mcC_0^{G/P}\hookrightarrow \mathbb P(\mfp_{-1})$\\
 \hline
 &&&&\\
 $B_n$, $n\geq 3$& \dynkin[parabolic=2]{B}{}  &
 \begin{dynkinDiagram}{B}{xxx*.**}\end{dynkinDiagram} & $A_1\times B_{n-2}$ & $\mathbb P^1\times Q^{2n-5}\hookrightarrow \mathbb{P}^{4n-7}$ Segre\\
&&&&\\
 $D_n$, $n\geq 5$  &\dynkin[parabolic=2]{D}{} &\begin{dynkinDiagram}{D}{xxx*.***}\end{dynkinDiagram}& $A_1\times D_{n-2}$ &  $\mathbb P^1\times Q^{2n-6}\hookrightarrow \mathbb{P}^{4n-9}$ Segre  \\
 &&&&\\
 $D_4$ &\dynkin[parabolic=2]{D}{4}&\begin{dynkinDiagram}{D}{xxxx}\end{dynkinDiagram}& $A_1\times D_{2}$ &  $\mathbb P^1\times Q^{2}\hookrightarrow \mathbb{P}^{7}$ Segre\\
 &&&&\\
$E_6$&   \dynkin[parabolic=2]{E}{6}   &\begin{dynkinDiagram}{E}{*x*x***}
 \end{dynkinDiagram}
&$A_5$& $\textrm{Gr}(3,3)\hookrightarrow \mathbb{P}^{19}$ by $\mcO(1)$ \\
 &&&&\\
 $E_7$&  \dynkin[parabolic=1]{E}{7}    &\begin{dynkinDiagram}{E}{x*x****}
 \end{dynkinDiagram}
 &$D_6$& $\mathbb {S}^6\hookrightarrow\mathbb{P}^{31}$ by $\mcO(1)$ \\
  &&&&\\
 $E_8$&  \begin{dynkinDiagram}{E}{*******x}
 \end{dynkinDiagram}
 &\begin{dynkinDiagram}{E}{******xx}
 \end{dynkinDiagram}
 &$E_7$& $\mcC_0^{G/P}\hookrightarrow\mathbb{P}^{55}$ by $\mcO(1)$; exceptional\\
  &&&&\\
 $F_4$&  \begin{dynkinDiagram}{F}{x***}
 \end{dynkinDiagram}
   &\begin{dynkinDiagram}{F}{xx**}
 \end{dynkinDiagram}
 &$C_3$& $\textrm {LGr}(3,3)\hookrightarrow\mathbb{P}^{13}$ by $\mcO(1)$\\
  &&&&\\
 $G_2$&  \begin{dynkinDiagram}{G}{x*}
 \end{dynkinDiagram}
   &\begin{dynkinDiagram}{G}{xx}
 \end{dynkinDiagram}
 &$A_1$& $\mathbb {P}^1\hookrightarrow\mathbb{P}^{3}$ by $\mcO(3)$; twisted cubic curve\\
 \hline
\end{tabular}
\caption{Adjoint varieties corresponding to long simple roots and their VMRT-structures}
\label{tab3}
\end{table}

\end{document}